\documentclass{amsart}
\usepackage{enumerate}
\usepackage{graphicx}
\usepackage{amsmath, amsthm, amssymb}
\usepackage{tikz,ifthen,calc}
\usetikzlibrary{graphs,quotes,shapes.geometric, arrows}
\usepackage[capitalize]{cleveref}
\usepackage{url}
\usepackage{multirow,multicol}
\usepackage[margin=1in]{geometry}

\renewcommand{\part}[1] {\vspace{.10in} {\bf (#1)}}
\newtheorem{theorem}{Theorem}[section]
\newtheorem{definition}[theorem]{Definition}
\newtheorem{lemma}[theorem]{Lemma}
\newtheorem{proposition}[theorem]{Proposition}
\newtheorem{corollary}[theorem]{Corollary}

\newtheorem{remark}[theorem]{Remark}
\newtheorem{result}[theorem]{Result}
\numberwithin{equation}{section}

\newcommand{\vo}{\ensuremath{\mathbf{1}}}
\newcommand{\vz}{\ensuremath{\mathbf{0}}}
\newcommand{\vc}{\ensuremath{\mathbf{c}}}

\newcommand{\vu}{\ensuremath{\mathbf{u}}}
\newcommand{\vv}{\ensuremath{\mathbf{v}}}

\newcommand{\vs}{\ensuremath{\mathbf{s}}}
\newcommand{\vx}{\ensuremath{\mathbf{x}}}
\newcommand{\vy}{\ensuremath{\mathbf{y}}}

\newcommand{\bb}{\mathbb}
\newcommand{\Z}{\bb{Z}}
\newcommand{\R}{\bb{R}}
\newcommand{\C}{\bb{C}}
\newcommand{\cB}{\mathcal{B}}

\newcommand{\BSHM}{{\rm BSHM}}
\newcommand{\rank}{{\rm rank}}

\newcommand{\Row}{{\rm Row}}

\newcommand{\wh}{\widehat}

\newcommand{\la}{\lambda}

\long\def\symbolfootnote[#1]#2{\begingroup%
\def\thefootnote{\fnsymbol{footnote}}\footnote[#1]{#2}\endgroup}

\begin{document}

\title{Constructions and restrictions for balanced splittable 
Hadamard matrices}

\author{Jonathan Jedwab \and Shuxing Li \and Samuel Simon}

\date{7 July 2022 (revised 2 February 2023)}

\symbolfootnote[0]{
Department of Mathematics,
Simon Fraser University, 8888 University Drive, Burnaby BC V5A 1S6, Canada.
\par
J.~Jedwab is supported by an NSERC Discovery Grant.
S.~Li is supported by a PIMS Postdoctoral Fellowship.
\par
Email: {\tt jed@sfu.ca}, {\tt shuxing\_li@sfu.ca}, {\tt samuel\_simon@sfu.ca}
}

\begin{abstract}
A Hadamard matrix is balanced splittable if some subset of its rows has the property that the dot product of every two distinct columns takes at most two values. This definition was introduced by Kharaghani and Suda in 2019, although equivalent formulations have been previously studied using different terminology.
We collate previous results phrased in terms of balanced splittable Hadamard matrices, real flat equiangular tight frames, spherical two-distance sets, and two-distance tight frames.
We use combinatorial analysis to restrict the parameters of a balanced splittable Hadamard matrix to lie in one of several classes, and obtain strong new constraints on their mutual relationships. An important consideration in determining these classes is whether the strongly regular graph associated with the balanced splittable Hadamard matrix is primitive or imprimitive.
We construct new infinite families of balanced splittable Hadamard matrices in both the primitive and imprimitive cases.
A rich source of examples is provided by packings of partial difference sets in elementary abelian $2$-groups, from which we construct Hadamard matrices admitting a row decomposition so that the balanced splittable property holds simultaneously with respect to every union of the submatrices of the decomposition.

\end{abstract}

\maketitle

\begin{section}{Introduction}\label{Sec:intro}

Hadamard matrices are one of the central topics of combinatorial design theory, with connections to symmetric designs, orthogonal arrays, transversal designs, and regular two-graphs. Hadamard matrices have many practical applications in experimental design, spectroscopy, error correction, signal modulation and separation, signal correlation, and cryptography \cite{Horadam}.
We write $I_n$ for the $n \times n$ identity matrix.  

\begin{definition} \label{Defn:HM}
An $n \times n$ matrix $H$ over $\{1,-1\}$ is a \emph{Hadamard matrix of order $n$} if $H^T H=nI_n$.
\end{definition}

Suppose $H$ is a Hadamard matrix of order $n$. 
By definition, the columns of $H$ are pairwise orthogonal. Since $H$ is square and $\frac{1}{n}H^T$ is a left inverse for $H$, we have $\frac{1}{n}HH^T=I_n$ and so the rows of $H$ are also pairwise orthogonal. 
It is straightforward to show that $n=1,2$ or $n \equiv 0 \pmod{4}$. The Hadamard matrix conjecture, proposed by Paley \cite{Paley} in 1933, states that these necessary conditions are also sufficient. The conjecture is known to hold for $n < 668$ \cite{KTR}, but a proof remains elusive. 

Several authors 
impose additional structure on a Hadamard matrix in order to better understand the existence pattern and make connections with other combinatorial configurations \cite{Handbook,Will}.
The following additional structure was proposed in 2019 by Kharaghani and Suda \cite{KS}, who established connections with strongly regular graphs, equiangular lines, mutually unbiased Hadamard matrices, and commutative association schemes \cite[Sections 2 and 5]{KS}. 

\begin{definition}
\label{Defn:bshm}
Let $H$ be a Hadamard matrix of order $n$, let $H_1$ be an $\ell \times n$ submatrix of~$H$ where $1 \le \ell \le n-1$, 
and let $a, b$ be integers. Then $H$ is a \emph{balanced splittable Hadamard matrix} with respect to $H_1$ if the dot product of every two distinct columns of $H_1$ lies in~$\{a,b\}$. In this case, we say that $H$ is a $\BSHM(n,\ell,a,b)$ with respect to~$H_1$.
\end{definition}

We prefer ``balanced splittable'' to the expression ``balancedly splittable'' used in \cite{KS}.

The central question in the study of balanced splittable Hadamard matrices is to determine
for which parameter sets $(n,\ell,a,b)$ there exists a $\BSHM(n,\ell,a,b)$ with respect to some $\ell \times n$ submatrix.
A secondary question (not considered in this paper) is to determine whether a specified order $n$ Hadamard matrix is a $\BSHM(n,\ell,a,b)$ with respect to some $\ell \times n$ submatrix and for some $a,b$.

The remainder of this paper is organized as follows.
In \cref{Sec:prev}, we review some elementary results on balanced splittable Hadamard matrices. We then summarize the results given by Kharaghani and Suda~\cite{KS}, and review further results originally expressed in the language of real flat equiangular tight frames and spherical two-distance sets.
In \cref{Sec:constraints}, we classify the parameters $(n,\ell,a,b)$ of a nontrivial $\BSHM(n,\ell,a,b)$: apart from those with $\ell \in \{2, n-2\}$ (characterized in \cref{Lemma:A1A2}~$(ii)$ and \cref{Cor:ell2}), there are five classes as summarized in \cref{Tab:summary} below.
In \cref{Sec:primitive-open}, we tabulate the parameter sets for the smallest open primitive cases.
In \cref{Sec:primitive}, we construct new infinite families of primitive balanced splittable Hadamard matrices via the character table of partial difference sets in elementary abelian $2$-groups. In particular, we determine for each parameter set with $n \in \{64,256\}$ whether a primitive $\BSHM(n,\ell,a,b)$ exists.
We then use packings of partial difference sets to produce infinite families of Hadamard matrices that have the balanced splittable property with respect to multiple disjoint submatrices simultaneously.
In \cref{Sec:imprimitive}, we construct new infinite families of imprimitive balanced splittable Hadamard matrices by means of a Kronecker product construction, and further constrain the possible parameter sets.
In \cref{Sec:future}, we propose some open questions for future research. 

\begin{table}[ht!]
\caption{Five classes for a $\BSHM(n,\ell,a,b)$ satisfying $2 < \ell < n-2$, and its associated 
strongly regular graph $G$ with parameters $(v,k,\lambda,\mu)$, up to application of 
the switching transformation \eqref{Eq:switching} and interchange of $a,b$ and (for $b = -a$) negation of columns.
(Types 1 and 2 are defined in \cref{Prop:SRG2}, the associated graph is defined in \cref{Def:assgraph}, and primitive and imprimitive are defined in \cref{Def:primBSHM}.)}
\begin{center}
\scriptsize
\begin{tabular}{|c||c|c|c|c|c|}
																									   \hline
  		& $b=-a$				& \multicolumn{4}{c|}{$b\ne -a$} 														\\ \hline
  		& 					& \multicolumn{2}{c|}{Type 1}						& \multicolumn{2}{c|}{Type 2}						\\ \hline
  		& $n > 2\ell$				& \multicolumn{2}{c|}{$n > 2\ell$} 					& \multicolumn{2}{c|}{$n \ge 2\ell$}					\\ \hline
    		& primitive				& imprimitive			& primitive				& imprimitive		& primitive 					\\ \hline
parameter	& $n=\frac{\ell^2-a^2}{\ell-a^2}$,	& $(n,\ell,a,b)=$  		& $n=\frac{(\ell-a)(\ell-b)}{\ell+ab}$,	& $(n,\ell,a,b)=$ 	& $n=\frac{(\ell-a)(\ell-b)}{\ell+ab-a-b}$,	\\ 
relations	& $\ell \equiv a \pmod{4}$, 		& $(4rs,4s-1,4s-1,-1)$  	& $\ell \equiv a \equiv b \pmod{4}$,	& $(8rs,4s,4s,0)$ 	& $\ell \equiv a \equiv b \pmod{4}$,		\\ 
		& $a > 0$ (even)			& for $r \ge 2$, $s \ge 1$  	& $a > 0 \ge b$				& for $r,s \ge 1$	& $a > 0 \ge b$					\\ \hline 
\multirow{6}{*}{$G$}
   		& $v=n$,				& \multirow{6}{*}{$4sK_r$}	& $v=n$,				& \multirow{6}{*}{$4sK_{2r}$}	
																			& $v=n$,					\\ 
		& $k=\frac{(n-1)a-\ell}{2a}+c$,		&				& $k=\frac{\ell-b+nb}{b-a}$,		& 			& $k=\frac{\ell-b+n(b-1)}{b-a}$,		\\
		& $\lambda=\frac{(n-4)a+n-4\ell}{4a}+c$,&				& $\lambda=\mu+\frac{2(\ell-b)-n}{b-a}$,& 			& $\lambda=\mu+\frac{2(\ell-b)-n}{b-a}$,	\\
		& $\mu=\frac{n(a-1)}{4a}+c$		&				& $\mu=\frac{nb(b+1)}{(b-a)^2}$		& 			& $\mu=\frac{nb(b-1)}{(b-a)^2}$			\\ 
		& for $c=0$ and $c=\frac{n}{2a}$	&				&					&			&						\\ [1ex] \hline
\multirow{2}{*}{integers} 		
		& \multirow{2}{*}{$\frac{\ell}{a}$ (odd), $\frac{n}{4a}$}
							&				& $\frac{\ell-b}{b-a}$,\, $\frac{n}{b-a}$,	
																& 			& $\frac{\ell-b}{b-a}$,\, $\frac{n}{b-a}$, 	\\ [1ex]
		& 					&				& $\frac{n(b+1)}{2(b-a)}$,\, $\frac{nb(b+1)}{(b-a)^2}$ 
																& 			& $\frac{n(b-1)}{2(b-a)}$,\, $\frac{nb(b-1)}{(b-a)^2}$	
																									\\ \hline
\end{tabular}
\end{center}
\normalsize
\label{Tab:summary}
\end{table}

\end{section}

\begin{section}{Previous results}\label{Sec:prev}

In this section, we summarize the previous state of knowledge for the existence of a $\BSHM(n,\ell,a,b)$.
\cref{Subsec:Elementary} contains some elementary results.
\cref{Subsec:KS} reviews the results presented by Kharaghani and Suda~\cite{KS}.
\cref{Subsec:ETF} contains results originally phrased in terms of real flat equiangular tight frames, and
\cref{Subsec:STS} contains results originally phrased in terms of spherical two-distance sets and two-distance tight frames.

We write $J_n$ for the $n \times n$ all-ones matrix, $0_{n \times m}$ for the $n \times m$ all-zeroes matrix, and $\vo$ for the all-ones column vector and $\vz$ for the all-zeroes column vector (with length determined by context).

\begin{subsection}{Elementary results}\label{Subsec:Elementary}

\begin{remark}
\label{Rem:A}
Suppose that $H$ is a $\BSHM(n,\ell,a,b)$ with respect to $H_1$.
The $(i,j)$ entry of $H_1^T H_1$ is the dot product of columns $i$ and $j$ of~$H_1$, so
\begin{equation}\label{Eq:Def}
    H_1^TH_1 = \ell I_n +aA+b(J_n-I_n-A)
\end{equation}
where $A = (A_{i,j})$ is the $n \times n$ symmetric matrix over $\{0,1\}$ with zero diagonal given by
\begin{equation}\label{Eq:Aij}
A_{i,j} =
\begin{cases}
1 & \mbox{if $i\ne j$ and the dot product of columns $i$ and $j$ of $H_1$ equals $a$,} \\
0 & \mbox{otherwise}.
\end{cases}
\end{equation}
By orthogonality of the rows of $H$, 
\begin{equation}\label{Eq:MMT}
H_1H_1^T=nI_\ell \quad \mbox{and} \quad H_2H_2^T=nI_{n-\ell} \quad \mbox{and} \quad H_1 H_2^T = 0_{\ell \times (n-\ell)}.
\end{equation}
\end{remark}

\begin{lemma}\label{Lemma:A1A2}
\mbox{}
\begin{enumerate}[$(i)$]
\item
A $\BSHM(n,\ell,a,b)$ with respect to a submatrix $H_1$ is also a $\BSHM(n,\ell,b,a)$ with respect to~$H_1$.

\item
The matrix $H = \begin{pmatrix}
    H_1 \\
    H_2
    \end{pmatrix}$ is a $\BSHM(n,\ell,a,b)$ with respect to $H_1$ if and only if $H$ is a  $\BSHM(n,n-\ell,-a,-b)$ with respect to $H_2$.

\item
The parameters of a $\BSHM(n,\ell,a,b)$ satisfy $|a|, |b| \le \min \{\ell,n-\ell\}$.
\end{enumerate}
\end{lemma}
\begin{proof}
$(i)$ and $(ii)$ and the relation $-\ell \le a, b \le \ell$ follow directly from \cref{Defn:bshm}. Then $(iii)$ follows from~$(ii)$.
\end{proof} 

\begin{remark}\label{Rem:basic}
\mbox{}
\begin{enumerate}[$(i)$]
\item
By \cref{Lemma:A1A2}~$(i)$, we may interchange the parameters $a$ and $b$ of a $\BSHM(n,\ell,a,b)$; we shall usually follow the convention of \cite{KS} by taking $a \ge b$.
By \cref{Lemma:A1A2}~$(ii)$, we may also apply the ``switching transformation''
\begin{equation}\label{Eq:switching}
H_1 \leftrightarrow H_2 \quad \mbox{and} \quad 
(\ell,a,b) \leftrightarrow (n-\ell,-a,-b),
\end{equation}
and we shall usually do so in order that $\ell \le n/2$ holds.

\item
We can easily characterize the case $\ell = 1$:
a Hadamard matrix $\begin{pmatrix} H_1 \\ H_2 \end{pmatrix}$ 
of order~$n>2$ is a $\BSHM(n,1,a,b)$ with respect to $H_1$ if and only if either
$(a,b)=(1,1)$ and $H_1 = \pm \vo^T$, or else
$\{a,b\} = \{1,-1\}$ and $H_1 \ne \pm \vo^T$.
We therefore regard the case $\ell = 1$ (and by $(i)$ the case $\ell = n-1$) as \emph{trivial}. 
Throughout the rest of the paper we shall consider a $\BSHM(n,\ell,a,b)$ only in the nontrivial cases $1 < \ell < n-1$ (and then $n \equiv 0 \pmod{4}$).

\item
Two Hadamard matrices are \emph{equivalent} if one can be transformed into the other by permutation and negation of some rows and columns. The balanced splittable property of a Hadamard matrix is preserved under row and column permutation, and under row negation, but not necessarily under column negation (except in the case $b=-a$).
\end{enumerate}
\end{remark}

\end{subsection}

\begin{subsection}{Results from Kharaghani and Suda \cite{KS}}\label{Subsec:KS}

We firstly exclude the case $b=a$.

\begin{result}[{\cite[Proposition~2.4]{KS}}]\label{Res:KS}
Suppose that $H$ is a nontrivial $\BSHM(n,\ell,a,b)$. Then $b \ne a$.
\end{result}

We next distinguish the cases $b=-a$ and $b \ne -a$.
The following result, for the case $b=-a$, follows from \cite[Propositions~2.13, 2.14, 2.16, 2.17]{KS}.
A Hadamard matrix $H$ of order~$n$ is \emph{regular} if $\vo^TH=c\sqrt{n}\vo^T$  for $c \in \{1,-1\}$.
Two Hadamard matrices $H$ and $L$ of order $n$ are \emph{unbiased} if 
each entry of $HL^T$ lies in $\big\{\sqrt{n}, -\sqrt{n}\big\}$.

\begin{result}\label{Res:b=-a}
Suppose $H = \begin{pmatrix} H_1 \\ H_2 \end{pmatrix}$ is a nontrivial
$\BSHM(n,\ell,a,-a)$ with respect to $H_1$. Then:
\begin{enumerate}[$(i)$]
\item $a$ is even and $\ell \equiv a \pmod{4}$

\item If $(n,\ell)=(4a^2, 2a^2 \pm a)$, then $H$ is equivalent to a regular Hadamard matrix and $L:=\frac{1}{2a}(H_1^TH_1-H_2^TH_2)$
is a Hadamard matrix such that $L$ and $H$ are unbaised. Conversely, if $L$ is a Hadamard matrix, then $(n,\ell)=(4a^2, 2a^2 \pm a)$.

\end{enumerate}

\end{result}

The following constructions for $b \ne -a$ are given by \cite[Theorems~3.1, 3.2, 3.3, 3.4, 3.7, 3.8]{KS} and \cref{Lemma:A1A2}~$(ii)$.
Those described in \cref{Res:Constructions}~$(i)$--$(iv)$ involve the Kronecker product of Hadamard matrices, 
whereas those described in \cref{Res:Constructions}~$(v)$,$(vi)$ are direct constructions.
A Hadamard matrix of the form $H = C + I_{q+1}$, where $C^T = -C$, is called \emph{skew-type}. 

\begin{result}\label{Res:Constructions}
Suppose there exist Hadamard matrices of orders $n$ and~$s$. Then there exists:
\begin{enumerate}[$(i)$]
    \item a $\BSHM(n^2,2n-2,n-2,-2)$ for $n \ge 2$
    \item a $\BSHM(n^2,2n-1,n-1,-1)$ 
	  for $n \ge 4$
    \item a $\BSHM(ns,n,n,0)$
 	  for $n \ge 2$
    \item a $\BSHM(2^{2m},2^{m-1}(2^m - 1),2^{m-1},-2^{m-1})$ for $m \ge 2$
    \item a $\BSHM(n,2,2,0)$ 
	  for $n \ge 4$
    \item a $\BSHM(q(q+1),q,q,-1)$ for $q \ge 3$, where $q+1$ is the order of a skew-type Hadamard matrix.
\end{enumerate}
\end{result}

A Hadamard matrix can have the balanced splittable property with respect to multiple disjoint submatrices simultaneously. 
For an $m \times n$ matrix $A = (a_{ij})$ and a matrix $B$, write $A \otimes B$ for the Kronecker product 
\[
\begin{pmatrix} 
a_{11}B & \dots  & a_{1n}B  	\\
\vdots	& \ddots & \vdots	\\
a_{m1}B & \dots  & a_{mn}B
\end{pmatrix}.
\]
Let $S_1 = \begin{pmatrix} 1 & 1 \\ 1 & -1 \end{pmatrix}$, and define $S_r = S_1 \otimes S_{r-1}$ recursively for $r \ge 2$. Then $S_r$ is the \emph{Sylvester-type} Hadamard matrix of order~$2^r$. The following refinement of \cref{Res:Constructions}~$(iv)$ is given by \cite[Theorem 3.7]{KS}; we shall greatly extend this construction in \cref{Cor:PDSpacking}.

\begin{result}\label{Res:twin}
The Sylvester-type Hadamard matrix $H$ of order $2^{2m}$ can be partitioned into submatrices $H_1, H_2, H_3$ of size 
$2^m \times 2^{2m}$,\, $2^{m-1}(2^m-1) \times 2^{2m}$,\, $2^{m-1}(2^m-1) \times 2^{2m}$, respectively,
such that $H$ is simultaneously 
a $\BSHM(2^{2m},2^m,2^m,0)$ with respect to $H_1$, and is
a $\BSHM(2^{2m}, 2^{m-1}(2^m-1), 2^{m-1}, -2^{m-1})$ with respect to $H_2$ and with respect to~$H_3$.
\end{result}

The following result on the existence of a $\BSHM(16,6,2,-2)$ is given by \cite[Remark 2.15, Examples 2.18, 2.19, 2.20]{KS}.

\begin{result}\label{Res:16}
Exactly three of the five inequivalent Hadamard matrices of order $16$ are equivalent to a $\BSHM(16,6,2,-2)$.
\end{result}

The following nonexistence results are derived in \cite[Proposition~2.21]{KS} using careful analysis and computer search. We shall recover and extend these results theoretically in \cref{Cor:mod4}.

\begin{result}\label{Res:36}
There is no $\BSHM(36,\ell,a,b)$ for the following parameters:
\begin{enumerate}[$(i)$]
    \item $(\ell,a,b) = (10,4,-2)$.
    \item $(\ell,a,b) = (25,1,-5)$.
    \item $(\ell,a,b) = (14,2,-4)$.
    \item $(\ell,a,b) = (20,2,-4)$.
\end{enumerate}
\end{result}

\end{subsection}

\begin{subsection}{Real flat equiangular tight frames}\label{Subsec:ETF}

Equiangular tight frames are widely studied in communications, coding theory, and sparse approximation \cite{DGS91,FMT,N,SH,STD+,W}. 

\begin{definition}\label{Def:ETF}
Let $S$ be an $\ell \times n$ matrix with complex entries and $S^*$ be the conjugate transpose of~$S$. 
The matrix $S$ is a \emph{tight frame} if $SS^*=nI_{\ell}$. 
A tight frame $S$ is \emph{real} if its entries are all real, and \emph{flat} if its entries all have magnitude~$1$ (and so is real flat if its entries all lie in $\{1,-1\}$).
A tight frame $S$ is an \emph{equiangular tight frame} (ETF) if all the off-diagonal entries of the Gram matrix $S^*S$ have constant magnitude. 
An $\ell \times n$ ETF is a \emph{Hadamard ETF} if it is a submatrix of an order $n$ Hadamard matrix.
\end{definition}

If $S$ is an $\ell \times n$ real flat ETF with columns $\vs_1,\dots,\vs_n$, then by considering the eigenvalues of the Gram matrix $S^*S$
one can derive that the dot product of $\vs_i$ and $\vs_j$ is
$\pm\sqrt{\frac{\ell(n-\ell)}{n-1}}$ for all distinct~$i,j$ \cite[Theorem 2.3]{SH}, \cite[Proposition~2]{STD+}.
Therefore a Hadamard matrix $H = \begin{pmatrix}H_1 \\ H_2 \end{pmatrix}$ is a $\BSHM(n,\ell,a,-a)$ with respect to $H_1$ if and only if $H_1$ is an $\ell \times n$ Hadamard ETF 
(and then $a = \pm\sqrt{\frac{\ell(n-\ell)}{n-1}}$), as noted in \cite[Proposition 2.12]{KS}.  
This implies the following necessary conditions for a $\BSHM(n,\ell,a,-a)$, as consequences of \cite[Theorems A, C and Corollary 14]{STD+}.

\begin{result}\label{Res:ETF}
Suppose there exists a nontrivial $\BSHM(n,\ell,a,-a)$. Then:
\begin{enumerate}[(i)]
\item 
$\ell \ne \frac{n}{2}$

\item 
$\sqrt{\frac{\ell(n-1)}{n-\ell}}$ and $\sqrt{\frac{(n-\ell)(n-1)}{\ell}}$ are odd integers 

\item 
$(n-2\ell)\sqrt{\frac{n-1}{\ell(n-\ell)}}$ is an integer

\item 
$n \le \min\{ \frac{\ell(\ell+1)}{2}, \frac{(n-\ell)(n-\ell+1)}{2} \}$.
\end{enumerate}
\end{result} 

We rephrase the elegant construction of a Hadamard ETF in \cite[Theorem 2]{FJMP} as the following result, which signficantly extends \cref{Res:Constructions}~$(iv)$.

\begin{result}\label{Res:Equi}
Suppose there exists a Hadamard matrix of order~$n$. Then there exists a $\BSHM(4n^2,2n^2-n,n,-n)$.
\end{result}

We next describe a characterization of a real flat ETF involving quasi-symmetric balanced incomplete block designs, and some resulting restrictions on the parameters of a $\BSHM(n,\ell,a,-a)$.

\begin{definition}
A $(v,k,\la,r,b)$ \emph{balanced incomplete block design} (BIBD) is a pair $(V,\cB)$, where $V$ is a set of $v$ points and $\cB$
is a collection of $k$-subsets (blocks) of $V$ with $|\cB|=b$, such that each element of $V$ is contained in exactly $r$ blocks and each $2$-subset of $V$ is contained in exactly $\la$ blocks. A $(v,k,\la,r,b,x,y)$ quasi-symmetric BIBD is a $(v,k,\la,r,b)$ BIBD for which every two distinct blocks intersect in either $x$ or $y$ points, where $0 \le x < y$.
\end{definition}

\begin{result}{\rm{\cite[Theorem 3]{FJMP}}}\label{Res:ChaEqui}
An $\ell \times n$ real flat ETF exists if and only if a $(v,k,\la,r,b,x,y)$ quasi-symmetric BIBD exists with parameters
\begin{align*}
a=\sqrt{\frac{\ell(n-\ell)}{n-1}}, \quad 
v=\ell, \quad
k=\frac{\ell-a}{2}, \quad \la=\frac{(n-1)(\ell-a)(\ell-a-2)}{4\ell(\ell-1)}, \\ r=\frac{(n-1)(\ell-a)}{2\ell}, \quad b=n-1, \quad x=\frac{\ell-3a}{4}, \quad y=\frac{\ell-a}{4}.
\end{align*}
In particular, if a $\BSHM(n,\ell,a,-a)$ exists then each of these expressions is a nonnegative integer.
\end{result}

\end{subsection}

\begin{subsection}{Spherical two-distance sets and two-distance tight frames}\label{Subsec:STS}

Many authors have studied finite sets of points over the unit sphere of~$\R^n$ having a small number of distinct distances among the points \cite{BBS,DGS91,GY};
when this number is two, the sets are called spherical two-distance sets \cite{BGOY,BY,B,CTT,DGS77,LRS,L,M,N}.

\begin{definition}\label{Defn:two-distance}
Let $S=\{ \vs_1, \dots, \vs_n \}$ be a subset of $\R^{\ell}$, where $||\vs_i||$ is constant over~$i$. The set $S$ is a \emph{spherical two-distance set} with values $a,b$
if $\vs_i \cdot \vs_j$ takes only the values $a$ and $b$ over all distinct $i,j$. A spherical two-distance set $S$ with values $a,b$ is \emph{regular} if the number of $j$, other than $i$, for which $\vs_i \cdot \vs_j = a$ holds is constant over~$i$.
\end{definition}

It follows from \cref{Defn:two-distance} that if
$H = \begin{pmatrix}H_1 \\ H_2 \end{pmatrix}$ is a $\BSHM(n,\ell,a,b)$ with respect to $H_1$, then
the columns of $H_1$ form a spherical two-distance set over $\{1,-1\}^{\ell}$ with values~$a,b$. 
We shall see in \cref{Prop:SRG1} that a $\BSHM(n,\ell,a,-a)$ can be transformed, by negating columns as necessary, so that its corresponding spherical two-distance set is regular; and in \cref{Prop:SRG2} that the spherical two-distance set corresponding to a $\BSHM(n,\ell,a,b)$ with $b \ne -a$ is necessarily regular.
We thereby obtain the following constraints from the cited results on spherical two-distance sets. 

\begin{result}\label{Res:NonEqui}
\mbox{}
\begin{enumerate}[$(i)$]
\item A $\BSHM(n,\ell,a,b)$ with $n > 2\ell+1 \ge 5$ and $\ell>a>b$ satisfies $\frac{\ell-a}{\ell-b}=\frac{m-1}{m}$ for some integer $m \ge 2$. In particular, when $b=-a$ we have $\frac{\ell}{a}=2m-1$ \emph{(\cite[Theorem 2]{LRS}, \cite[Theorem 2]{N})}.

\item A $\BSHM(n,\ell,a,b)$ with $\ell>\max\{a,b\}$ satisfies
\begin{equation*}
n \le \begin{cases}
        5 & \mbox{if $\ell=2$,} \\
        6 & \mbox{if $\ell=3$,} \\
        10 & \mbox{if $\ell=4$,} \\
        16 & \mbox{if $\ell=5$,} \\
        27 & \mbox{if $\ell=6$,} \\
        \frac{\ell(\ell+1)}{2} & \mbox{if $\ell \ge 7$ and $\ell \ne (2k+1)^2-3$ for each $k \ge 1$,} \\
        \frac{\ell(\ell+3)}{2} & \mbox{if $\ell \ge 7$ and $\ell = (2k+1)^2-3$ for some $k \ge 1$}
      \end{cases}    
\end{equation*}
\emph{(\cite[Theorem 1.1]{BY},\cite[Theorem 1]{GY})}.

\item A $\BSHM(n,\ell,a,b)$ with $\ell>\max\{a,b\}$ satisfies $ab \le 0$ \emph{(\cite[Theorem 2.14]{CTT})}.
\end{enumerate}
\end{result}

A \emph{two-distance tight frame} is a real flat tight frame whose columns form a spherical two-distance set.
A two-distance tight frame with values $a,b$ where $b=-a$ is a real flat ETF, as characterized in \cref{Res:ChaEqui}.
A two-distance tight frame with values $a,b$ where $b \ne -a$ will be characterized in \cref{Thm:ChaNonEqui}, subject to an additional condition, using the following definition.
(\cref{Thm:ChaNonEqui} is not a previous result, but is included as a counterpart to \cref{Res:ChaEqui}.)

\begin{definition}
A $(v,\{k_1,k_2\},\la)$ pairwise balanced design is a pair $(V,\cB)$ where $V$ is a set of $v$ points and $\cB$
is a collection of subsets (blocks) each of size $k_1$ or $k_2$, such that every $2$-subset of $V$ is contained in exactly $\la$ blocks. 
The \emph{incidence matrix} of a $(v,\{k_1,k_2\},\la)$ pairwise balanced design is the $v \times |\cB|$ matrix whose $(i,j)$ entry is $1$ if point $i$ is contained in block $j$, and is $0$ otherwise. 
\end{definition}

The following characterization of a two-distance tight frame $S$ assumes that the first column of $S$ is $\vo$ and that $S \vo = \vz$. 
The first assumption is readily satisfied, by negating rows of $S$ as necessary.
We shall show in \cref{Prop:SRG2} that for $b \ne -a$, a nontrivial $\BSHM(n,\ell,a,b)$ with respect to a submatrix $H_1$ can be transformed,
using the switching transformation \eqref{Eq:switching} if necessary, so that $H_1 \vo = \vz$;
we may then apply \cref{Thm:ChaNonEqui} with $S=H_1$.

\begin{theorem}\label{Thm:ChaNonEqui}
Let $S$ be an $\ell \times n$ matrix having first column $\vo$ and satisfying $S \vo = \vz$, and write
$S = \begin{pmatrix} \vo & J-2X \end{pmatrix}$ where $J$ is the $\ell \times (n-1)$ all-ones matrix. 
Then $S$ is a two-distance tight frame with values $a,b$ if and only if
$X$ is the incidence matrix of an $(\ell,\{\frac{\ell-a}{2}, \frac{\ell-b}{2}\},\frac{n}{4})$ pairwise balanced design with $n-1$ blocks for which 
every two distinct blocks have intersection size lying in $\{ \frac{\ell-a}{4}, \frac{\ell-b}{4}, \frac{\ell+a-2b}{4}, \frac{\ell+b-2a}{4} \}$. 
\end{theorem}
\begin{proof}
The $\ell \times n$ matrix $S$ is over $\{1, -1\}$ if and only if the $\ell \times (n-1)$ matrix $X$ is over~$\{0,1\}$.
and let $V$ and $\cB$ be the point set and block set corresponding to~$X$.

Let $X$ have rows $\vy_1^T, \dots, \vy_\ell^T$.
Since $S \vo = \vz$, we have $\vy_i \cdot \vo = \frac{n}{2}$ for each $i$.
The rows of $S$ are pairwise orthogonal if and only if $(\vo-2\vy_i) \cdot (\vo-2\vy_j) = -1$ for all distinct $i,j$, 
which is therefore equivalent to $\vy_i \cdot \vy_j = \frac{n}{4}$ for all distinct $i,j$,
or equivalently every $2$-subset of $V$ is contained in exactly $\frac{n}{4}$ blocks.

Let $X$ have columns $\vx_1,\dots, \vx_{n-1}$, so the columns of $S$ are $\vo, \vo-2\vx_1, \dots, \vo-2\vx_{n-1}$. 
The dot product of columns $\vo$ and $\vo-2\vx_i$ of $S$ lies in $\{a,b\}$ for all $i$ if and only if
\begin{equation}\label{Eq:vxivo}
2\vx_i \cdot \vo \in \{\ell-a, \ell-b\} \quad \mbox{for each $i$},
\end{equation}
or equivalently each block of $\cB$ has size $\frac{\ell-a}{2}$ or $\frac{\ell-b}{2}$.
The dot product of columns $\vo -2\vx_i$ and $\vo - 2\vx_j$ of $S$ lies in $\{a,b\}$ for all distinct $i,j$ if and only if 
\[
(\vo - 2\vx_i) \cdot (\vo-2\vx_j) \in \{a,b\} \quad \mbox{for all distinct $i,j$},
\]
which using \eqref{Eq:vxivo} is equivalent to
\[
4 \vx_i \cdot \vx_j \in \{\ell-a, \ell-b, \ell+a-2b, \ell+b-2a\} \quad \mbox{for all distinct $i,j$},
\]
or equivalently every two distinct blocks of $\cB$ have intersection size lying in \\ $\{ \frac{\ell-a}{4}, \frac{\ell-b}{4}, \frac{\ell+a-2b}{4}, \frac{\ell+b-2a}{4} \}$. 
\end{proof}

\end{subsection}

\end{section}

\begin{section}{Parameter constraints for a nontrivial $\BSHM(n,\ell,a,b)$} \label{Sec:constraints}
In this section, we restrict the parameters of a nontrivial $\BSHM(n,\ell,a,b)$ and obtain strong new constraints on their mutual relationships. Apart from the case $\ell=2$, whose parameters are characterized completely in \cref{Cor:ell2}, we identify five parameter classes in \cref{Subsec:classification} as summarized in \cref{Tab:summary}.
A key step in the analysis is to significantly simplify and extend the result of \cite[Proposition~2.6]{KS}, as \cref{Prop:SRG1} for $b = -a$ and as \cref{Prop:SRG2} for $b \ne -a$. The parameter relations $b=-a$ and $b \ne -a$ differ fundamentally as a result of the observation in \cref{Rem:basic}~$(iii)$. We obtain additional new parameter relations in \cref{Subsec:further}.

\begin{subsection}{Associated Strongly Regular Graph}

Strongly regular graphs are among the most well-studied graphs in algebraic graph theory~\cite{BH}. 
We shall associate a $\BSHM(n,\ell,a,b)$ with a strongly regular graph if $b\ne-a$, and with two strongly regular graphs if $b = -a$.

\begin{definition}
A \emph{strongly regular graph} with parameters $(v, k, \lambda, \mu)$ is a graph with $v$ vertices, each of degree $k$, such that every two adjacent vertices have exactly $\lambda$ common neighbors and every two non-adjacent vertices have exactly $\mu$ common neighbors. 
\end{definition}

\begin{remark}\label{Rem:SRG}
\mbox{}
\begin{enumerate}[$(i)$]
\item Let $A$ be the adjacency matrix of a $(v, k, \lambda, \mu)$ strongly regular graph. Then
\[
A^2=kI_v+\lambda A+\mu(J_v-I_v-A).
\]

\item If $G$ is a $(v,k,\lambda,\mu)$ strongly regular graph, then its complement $\overline{G}$ is a $(v,\overline{k},\overline{\lambda},\overline{\mu})$ strongly regular graph, where
\begin{equation}\label{eqn:oGparams}
(\overline{k},\,\overline{\lambda},\,\overline{\mu}) = (v-k-1,\, v-2k+\mu-2,\, v-2k+\lambda).
\end{equation}
\end{enumerate}
\end{remark}

\begin{definition}
\label{Def:assgraph}
Let $H$ be a $\BSHM(n,\ell,a,b)$ with respect to a submatrix~$H_1$. The graph $G$ associated with $H$ has vertex set $\{1,2,\ldots,n\}$, and vertices $i$ and $j$ are adjacent when the $i^\text{th}$ and $j^\text{th}$ columns of $H_1$ have dot product~$a$.
\end{definition}

Suppose $H = \begin{pmatrix}
    H_1 \\
    H_2
    \end{pmatrix}$ is a $\BSHM(n,\ell,a,b)$ with respect to~$H_1$, 
and let $G$ be its associated graph.
Interchanging $a$ and $b$ according to \cref{Lemma:A1A2}~$(i)$
interchanges $G$ with its complement~$\overline{G}$, whereas
applying the switching transformation \eqref{Eq:switching}
according to \cref{Lemma:A1A2}~$(ii)$ leaves $G$ unchanged.

We shall use the following definition in \cref{Prop:SRG1,Prop:SRG2}.

\begin{definition}
\label{Defn:kai}
Let $H$ be a $\BSHM(n,\ell,a,b)$ with respect to a submatrix~$H_1$. 
Write $k_a(i)$ for the number of columns in~$H_1$, other than column $i$, whose dot product with column $i$ is~$a$.
\end{definition}

\end{subsection}

\begin{subsection}{The case $\boldsymbol{b = -a}$}

By \cref{Rem:basic}~$(iii)$, we may negate columns of a $\BSHM(n,\ell,a,-a)$ in order to produce an all-ones row. This leads to the structural result of \cref{Prop:SRG1}.

\begin{proposition}\label{Prop:SRG1}
Suppose $H = \begin{pmatrix}H_1 \\ H_2 \end{pmatrix}$ is a nontrivial $\BSHM(n,\ell,a,-a)$ with respect to~$H_1$. Then:

\begin{enumerate}[$(1)$]
\item
$n(\ell-a^2) = \ell^2-a^2$

\item
$a$ is even and $\frac{\ell}{a}$ is an odd integer 
and $\frac{n}{4a}$ is an integer

\item
$H$ can be transformed, by negating columns as necessary, to a matrix 
$H' = \begin{pmatrix}H'_1 \\ \vo^T \\ H'_2 \end{pmatrix}$ where $H'_1$ has size $\ell \times n$, and then

\begin{enumerate}[$(i)$]
\item
$H'$ is a $\BSHM(n,\ell,a,-a)$ with respect to $H'_1$, and $H'_1 \vo=\vz$
\item
$k_a(i)$ (defined in relation to $H'_1$) is a constant $k'_a$ independent of $i$, 
and the graph $G$ associated with $H'$ is strongly regular with parameters 
\[
(v,k,\lambda,\mu) = (n,k'_a,\lambda',\mu') = 
\Big(n,\, \frac{(n-1)a-\ell}{2a}, \, \frac{n-4}{4}+\frac{n-4\ell}{4a}, \, \frac{n(a-1)}{4a} \Big)
\]
\end{enumerate}

\item
$H$ can be transformed, by negating columns as necessary, to a matrix 
$H'' = \begin{pmatrix}H''_1 \\ \vo^T \\ H''_2 \end{pmatrix}$ where $H''_1$ has size $(\ell-1) \times n$, and then

\begin{enumerate}[$(i)$]
\item
$H''$ is a $\BSHM(n,\ell,a,-a)$ with respect to $\begin{pmatrix} H''_1 \\ \vo^T \end{pmatrix}$, and $H''_2 \vo=\vz$
\item
$k_a(i)$ (defined in relation to $\begin{pmatrix} H''_1 \\ \vo^T \end{pmatrix}$) is a constant $k''_a$ independent of $i$, 
and the graph $G$ associated with $H''$ is strongly regular with parameters 
\begin{align*}
(v,k,\lambda,\mu) =&\Big(n,\, k'_a + \frac{n}{2a}, \, \lambda'+\frac{n}{2a}, \, \mu'+\frac{n}{2a}\Big) \\
                             =&\Big(n,\,  \frac{(n-1)a+n-\ell}{2a}, \, \frac{n-4}{4}+\frac{3n-4\ell}{4a}, \, \frac{n(a+1)}{4a} \Big).
\end{align*}
\end{enumerate}

\end{enumerate}
\end{proposition}

\begin{proof}
We shall prove that (1), (2), (3) hold when $H$ is transformed to $H'$;
the proof that (1), (2), (4) hold when $H$ is transformed to $H''$ is similar.

Negation of a column of $H$ leaves the dot product of distinct columns of the upper $\ell \times n$ submatrix in $\{-a,a\}$, and so $H'$ is a $\BSHM(n,\ell,a,-a)$ with respect to $H'_1$. Row orthogonality in $H'$ then gives $H'_1 \vo=\vz$. This proves~(3)$(i)$.

Let $H'_1=(h_{st})$ and let the $i^\text{th}$ column of $H'_1$ be~$\vc_i$. For each $i$, we have
\[
\sum_{j=1}^n \vc_i \cdot \vc_j = k_a(i)a+(n-1-k_a(i))(-a)+\ell 
\]
by \cref{Defn:kai} for $k_a(i)$, and also
\[
\sum_{j=1}^n \vc_i \cdot \vc_j 
= \sum_{j=1}^n\sum_{s=1}^{\ell} h_{si}h_{sj} =\sum_{s=1}^{\ell}h_{si}\sum_{j=1}^nh_{sj} =0
\]
using $H'_1 \vo = \vz$.
Equate these two expressions to show that $k_a(i)=k'_a$ is independent of $i$ and that 
\begin{equation}\label{Eq:ka}
k'_a = \frac{(n-1)a-\ell}{2a}.
\end{equation}
Similarly
\begin{align*}
\sum_{j=1}^n (\vc_i \cdot \vc_j)^2&=k'_aa^2+(n-1-k'_a)(-a)^2+\ell^2 =\sum_{j=1}^n\sum_{s=1}^{\ell} h_{si}h_{sj} \sum_{t=1}^{\ell} h_{ti}h_{tj} \\
&=\sum_{s=1}^{\ell}\sum_{t=1}^{\ell} h_{si}h_{ti} \sum_{j=1}^nh_{sj}h_{tj}=n\sum_{s=1}^{\ell} h_{si}^2=n\ell,
\end{align*}
using row orthogonality in $H'$ to show that $\sum_{j=1}^n h_{sj}h_{tj} = 0$ for $s \ne t$.
Therefore (1) holds. 

Let $A$ be the adjacency matrix of $G$. 
By \eqref{Eq:Def} and \eqref{Eq:Aij} and \cref{Def:assgraph}, we have
\begin{equation}\label{H1'TH1'}
H_1'^TH_1' = \ell I_n+a A-a(J_n-I_n-A)=(\ell+a)I_n+2aA-aJ_n.
\end{equation}
Square both sides of \eqref{H1'TH1'} and simplify the resulting expressions using $H'_1H'^T_1 = nI_{\ell}$ from \eqref{Eq:MMT} and 
\eqref{H1'TH1'} again and $AJ_n = J_nA = k'_a J_n$ and \eqref{Eq:ka} and (1) to show that
\[
A^2=\frac{(n-1)a-\ell}{2a}I_n+\Big(\frac{n-4}{4}+\frac{n-4\ell}{4a}\Big)A+\frac{n(a-1)}{4a}(J_n-I_n-A).
\]
By \cref{Rem:SRG}~$(i)$, $G$ is therefore a strongly regular graph with the parameters $(v,k,\lambda,\mu)$ given in (3)$(ii)$, and the proof of (3) is complete.

It remains to prove~(2).
We have that $a$ is even by \cref{Res:b=-a}~$(i)$. 
Since $k'_a$ is an integer and $n \equiv 0 \pmod{4}$, we see from $\eqref{Eq:ka}$ that $\frac{\ell}{a}$ 
is an odd integer. 
Since $\lambda$ is an integer, 
$\frac{n}{4a}$ 
is then an integer.
\end{proof}

\begin{remark}\label{Rem:SRG1}
\mbox{}
\begin{enumerate}[$(i)$]
\item
\cref{Prop:SRG1} recovers \cref{Res:ETF}.
By \cref{Prop:SRG1}~(2), we have $\ell \ne \frac{n}{2}$.
We may assume that $a > 0$, by interchanging $a$ and $-a$ if necessary and noting from \cref{Prop:SRG1}~(1) that $a \ne 0$.
By \cref{Prop:SRG1}~(1), we can then simplify the quantities
$\sqrt{\frac{\ell(n-1)}{n-\ell}}$ and  
$\sqrt{\frac{(n-\ell)(n-1)}{\ell}}$ and 
$(n-2\ell)\sqrt{\frac{n-1}{\ell(n-\ell)}}$
as
$\frac{\ell}{a}$ and 
$\frac{n-\ell}{a}$ and
$\frac{n}{a}-\frac{2\ell}{a}$, respectively,
and the conditions on these quantities stated in \cref{Res:ETF} then follow from \cref{Prop:SRG1}~(2).

Furthermore, it follows from \cref{Prop:SRG1}~(1) that $0 < a^2 \le \ell-1$.
Since $a$ is even by \cref{Prop:SRG1}~(2), we obtain $4 \le a^2 \le \ell-1$.
Now $a^2 \ne \ell-1$, otherwise 
$\frac{\ell}{a} = a+\frac{1}{a}$ is not an integer in contradiction to \cref{Prop:SRG1}~(2). Therefore
$4 \le a^2 \le \ell-2$, so $n = \frac{\ell^2-a^2}{\ell-a^2} \le \frac{\ell^2-a^2}{2} \le \frac{\ell^2-4}{2}$ which implies the first inequality in \cref{Res:ETF}. The second inequality in \cref{Res:ETF} is then implied by interchanging $H_1$ and~$H_2$.

\item
\cref{Prop:SRG1} recovers the additional integer constraints described in \cref{Res:ChaEqui}.
By \cref{Prop:SRG1}~(2) (and also by \cref{Res:b=-a}~$(i)$), both $x$ and $y$ are integers.
By \cref{Prop:SRG1}~(1), we can write
$\lambda = \frac{n}{4}-\frac{n}{2a}+\frac{1}{2}(\frac{\ell}{a}-1)$ and
$r = (\frac{n}{a}-\frac{\ell}{a}) \frac{1}{2} (\frac{\ell}{a}-1)$, 
which are both integers by \cref{Prop:SRG1}~(2).

\end{enumerate}
\end{remark}

\end{subsection}

\begin{subsection}{The case $\boldsymbol{b \ne -a}$}

We shall prove the structural result of \cref{Prop:SRG2} for the case $b \ne -a$ (as a counterpart to \cref{Prop:SRG1} for the case $b=-a$).
We first establish two preliminary results. \cref{Prop:evalues} concerns the eigenspaces of the matrix $H_1^T H_1$, and \cref{Lemma:mod4} is an elementary result about the dot product of two vectors over~$\{1,-1\}$.

\begin{proposition}\label{Prop:evalues}
Suppose that $H = \begin{pmatrix} H_1 \\ H_2 \end{pmatrix}$ is a nontrivial $\BSHM(n,\ell,a,b)$ with respect to~$H_1$. Then the eigenvalues of $H_1^TH_1$ are $n$ with corresponding eigenspace $\Row(H_1)$, and $0$ with corresponding eigenspace $\Row( H_2)$.
\end{proposition}
\begin{proof}
By \eqref{Eq:MMT} we have $(H_1^T H_1) H_2^T = H_1^T (H_1 H_2^T) = 0_{n \times (n-\ell)}$, so the $n \times n$ matrix $H_1^T H_1$ has eigenvalue $0$ and the corresponding eigenspace contains $\Row(H_2)$.
We also have
$(H_1^T H_1) H_1^T = H_1^T (H_1 H_1^T) = H_1^T (n I_\ell) = n H_1^T$,
so $H_1^T H_1$ has eigenvalue $n$ and the corresponding eigenspace contains $\Row(H_1)$.
The result is now given by 
$\rank(H_1) = \rank(H_1 H_1^T) = \rank(n I_\ell) = \ell$
and similarly $\rank(H_2) = n-\ell$.
\end{proof}

Write $n_\vu$ for the number of $-1$ entries in a vector $\vu$ over~$\{1,-1\}$. 
\begin{lemma}\label{Lemma:mod4}
Let $\vu=(u_i)$, $\vv=(v_i)$ be vectors of length $\ell$ over~$\{1,-1\}$, and let $s$ be the number of~$i$ for which $u_i=v_i=-1$. Then
\begin{align*}
        \vu \cdot \vv = \ell - 2(n_\vu+n_\vv)+4 s \equiv \ell -2 (n_\vu + n_\vv) \pmod{4}.
\end{align*}
Furthermore, if $\vu \cdot \vv=\vu \cdot \vo=\vv \cdot \vo=0$, then $4 \mid \ell$ and $s=\frac{\ell}{4}$.
\end{lemma}

\begin{proof}
We calculate
\begin{align}
    \vu \cdot \vv &= s - (n_\vu - s) - (n_\vv-s)+(\ell-n_\vu-n_\vv+s) \nonumber \\
    &= \ell - 2(n_\vu+n_\vv)+4 s \label{Eq:vuvv} \\
    & \equiv \ell - 2(n_\vu+n_\vv) \pmod{4}. \nonumber
\end{align}

Now suppose $\vu \cdot \vo=\vv \cdot \vo=0$, so that $n_{\vu}=n_{\vv}=\frac{\ell}{2}$. Then \eqref{Eq:vuvv} shows that $\vu \cdot \vv = 0$ implies $s=\frac{\ell}{4}$.
\end{proof}

\begin{proposition}\label{Prop:SRG2}
Suppose $H = \begin{pmatrix}H_1 \\ H_2 \end{pmatrix}$ is a nontrivial $\BSHM(n,\ell,a,b)$ with respect to $H_1$, where $b \ne -a$.
Let $G$ be the graph associated with~$H$.
Then $k_a(i) = k_a$ is independent of $i$, and the matrix $H$ has exactly one of Types 1 and~2.

\begin{description}
\item[$H$ has Type 1]
    \mbox{}
    \begin{enumerate}[$({\rm A}1)$]
    \item
    $H_1 \vo = \vz$ and $(H_2 \vo)^T(H_2 \vo)=n^2$, and if $H_2$ does not contain the all-ones row~$\vo^T$ then $n^2$ is the sum of $n-\ell$ integer squares none of which is $n^2$
    \item
    $n(\ell+ab) = (\ell-a)(\ell-b)$ and $ab \le 0$
    \item 
    $\frac{\ell-b}{b-a}$ and $\frac{n}{b-a}$ and $\frac{n(b+1)}{2(b-a)}$ and $\frac{nb(b+1)}{(b-a)^2}$ are integers
    \item
    $k_a= \frac{\ell-b}{b-a} + \frac{nb}{b-a}$, and $G$ is strongly regular with parameters
    \[
    (v,k,\lambda,\mu) = 
     \Big(n, \, k_a, \, \frac{nb(b+1)}{(b-a)^2}+\frac{2(\ell-b)}{b-a}-\frac{n}{b-a}, \, \frac{nb(b+1)}{(b-a)^2}\Big).
    \]
    \end{enumerate}

\item[$H$ has Type 2]
    \mbox{}
    \begin{enumerate}[$({\rm B}1)$]
    \item
    $H_2 \vo = \vz$ and $(H_1 \vo)^T(H_1 \vo)=n^2$, and if $H_1$ does not contain the all-ones row~$\vo^T$ then $n^2$ is the sum of $\ell$ integer squares none of which is $n^2$
    \item
    $n(\ell+ab-a-b) = (\ell-a)(\ell-b)$ and $ab \le 0$
    \item 
    $\frac{\ell-b}{b-a}$ and $\frac{n}{b-a}$ and $\frac{n(b-1)}{2(b-a)}$ and $\frac{nb(b-1)}{(b-a)^2}$ are integers
    \item
    $k_a= \frac{\ell-b}{b-a} + \frac{n(b-1)}{b-a}$, and $G$ is strongly regular with parameters
    \[
    (v,k,\lambda,\mu) = 
     \Big(n, \, k_a, \, \frac{nb(b-1)}{(b-a)^2}+\frac{2(\ell-b)}{b-a}-\frac{n}{b-a}, \, \frac{nb(b-1)}{(b-a)^2}\Big).
    \]
    \end{enumerate}
\end{description}
Furthermore, 
$\begin{pmatrix}H_1 \\ H_2 \end{pmatrix}$ is a $\BSHM(n,\ell,a,b)$ of Type 1
with respect to~$H_1$ if and only if
$\begin{pmatrix}H_2 \\ H_1 \end{pmatrix}$ is a $\BSHM(n,n-\ell,-a,-b)$ of Type 2 with respect to~$H_2$, and the switching transformation \eqref{Eq:switching} maps between these two matrices. 
\end{proposition}

\begin{proof}
Relabel the rows and columns of $H$ if necessary so that the $(1,1)$ entry of $H$ is~$1$. 
Let the $i^\text{th}$ column of $H_1$ be~$\vc_i$. Using row orthogonality in $H$, a similar derivation to that in the proof of \cref{Prop:SRG1} gives
\[
\sum_{j=1}^n (\vc_i \cdot \vc_j)^2 = k_a(i)a^2+\big(n-1-k_a(i)\big)b^2 +\ell^2 = n\ell
\quad \mbox{for all $i$}.
\]
Therefore $k_a(i) = k_a$ is independent of $i$, and 
\begin{equation}
    k_a= \frac{\ell(n-\ell)-b^2(n-1)}{a^2-b^2} \label{Eq:Alpha}
\end{equation}
using that $b \ne -a$ by assumption and that $b \ne a$ by \cref{Res:KS}.

Let $A$ be the adjacency matrix of $G$, so that $A \vo = k_a \vo$ and
\begin{align}
    H_1^TH_1 &= \ell I_n+aA+b(J_n-I_n-A) \nonumber\\
             &=(\ell-b)I_n +(a-b)A+bJ_n \label{Eq:BaseEq}.
\end{align}
Multiplication on the right by $\vo$ gives
\begin{equation}\label{Eq:HH1}
    H_1^TH_1 \vo = c(\ell,a,b,k_a) \vo,
\end{equation}
where
\begin{equation}\label{Eq:c}
c(\ell,a,b,k_a) = \ell-b +(a-b)k_a +bn,
\end{equation}
and so $c(\ell,a,b,k_a)$ is an eigenvalue of~$H_1^T H_1$. Then by \cref{Prop:evalues}, 
\begin{equation}\label{Eq:either-or}
\mbox{$c(\ell,a,b,k_a) = 0$  or $n$}.
\end{equation}

By \cref{Lemma:A1A2}~$(ii)$, $H$ is a $\BSHM(n,\ell',a',b')$ with respect to~$H_2$, where $(\ell',a',b') = (n-\ell,-a,-b)$.
By \eqref{Eq:Alpha} and \eqref{Eq:c} we have
$c(\ell',a',b',k_{a'}) = n-c(\ell,a,b,k_a)$.
Therefore the transformation \eqref{Eq:switching} maps between the two cases specified in \eqref{Eq:either-or}.
Since the claimed conditions for $H$ having Type 2 are obtained from those for $H$ having Type 1 under the transformation \eqref{Eq:switching}, we may restrict attention to the case $c(\ell,a,b,k_a) = 0$ for the rest of the proof.
The final statement of the theorem will then follow from \cref{Lemma:A1A2}~$(ii)$.

Substitute $c(\ell,a,b,k_a) = 0$ in \eqref{Eq:HH1} to give $H_1^T H_1 \vo = \vz$. Multiply on the left by~$\vo^T$ to give
$(H_1 \vo)^T (H_1 \vo) = 0$, so that $H_1 \vo = \vz$.
Since $n^2= \vo^TH^TH\vo= \vo^T\left(H_1^TH_1+H_2^TH_2\right)\vo$, we then obtain $(H_2\vo)^T(H_2\vo)=n^2$. Therefore, if $H_2$ does not contain the all-ones row~$\vo^T$ then $n^2$ is the sum of $n-\ell$ integer squares none of which is $n^2$. This proves~(A1).

Substitute $c(\ell,a,b,k_a) = 0$ in \eqref{Eq:c} to give
\begin{equation}\label{Eq:simpka}
k_a= \frac{\ell-b}{b-a} + \frac{nb}{b-a}.
\end{equation}
Equate this expression with \eqref{Eq:Alpha} to give
\[
n(\ell+ab) = (\ell-a)(\ell-b).
\]
Rearrange as $ab(n-1) = \ell(\ell-a-b-n)$ and use the relation 
$\ell-a-b-n \le -b$ from \cref{Lemma:A1A2}~$(iii)$ and $\ell < n-1$ to show that $ab \le 0$.
This proves~(A2).

Square both sides of \eqref{Eq:BaseEq} and simplify the resulting expressions using $H_1H_1^T = nI_{\ell}$ and \eqref{Eq:BaseEq} again and $AJ_n = J_nA = k_a J_n$ and \eqref{Eq:simpka} and (A2) to show that
\[
A^2 = 
\Big(\frac{\ell-b}{b-a} + \frac{nb}{b-a}\Big)I_n +
\Big(\frac{2(\ell-b)}{b-a} - \frac{n}{b-a} +\frac{nb(b+1)}{(b-a)^2}\Big)A+
\frac{nb(b+1)}{(b-a)^2}(J_n-I_n-A).
\]
Therefore $G$ is a strongly regular graph with the parameters $(v,k,\lambda,\mu)$ given in~(A4), completing the proof of~(A4).

We claim that $\frac{n(b+1)}{2(b-a)}$ is an integer. 
Since each of the parameters $k, \lambda, \mu$ is an integer, 
this implies that $\frac{2n(b+1)}{b-a} + \lambda -\mu -2k = \frac{n}{b-a}$ is an integer, and then (A3) follows.

It remains to prove the claim.
Let $H_1 = (h_{st})$ and let $I = \{j : h_{1j} = -1\}$.
Since $H_1 \vo = \vz$ by (A1), each row of $H_1$ has dot product $0$ with $\vo$, and $|I| = \frac{n}{2}$.
For $s \ne 1$, row orthogonality in $H_1$ and \cref{Lemma:mod4} show that the number of
$j$ for which $(h_{1j}, h_{sj}) = (-1,-1)$ is $\frac{n}{4}$ and so the number of $j$ for which $(h_{1j}, h_{sj}) = (-1,1)$ is also~$\frac{n}{4}$.
Therefore
\begin{equation}\label{Eq:jinI}
\sum_{j \in I} h_{sj} = 
\sum_{j :\, h_{1j} = -1} h_{sj} = 
  \begin{cases}
  0				& \mbox{for $s \ne 1$}, \\
  -\frac{n}{2} 			& \mbox{for $s = 1$},
  \end{cases}
\end{equation}
and so
\begin{equation}\label{Eq:jinI2}
\sum_{j \in I} \vc_1 \cdot \vc_j = \sum_{j \in I} \sum_{s=1}^\ell h_{s1} h_{sj} = \sum_{s=1}^\ell h_{s1} \sum_{j \in I} h_{sj} = -\tfrac{n}{2}
\end{equation}
using \eqref{Eq:jinI} and the initial assumption that $h_{11} = 1$.

Now let
\[
I_a = \{j \in I : \vc_1 \cdot \vc_j =a \},
\]
and then
\[
\sum_{j \in I} \vc_1 \cdot \vc_j 
  = \sum_{j \in I_a} \vc_1 \cdot \vc_j +  
    \sum_{j \in I_b} \vc_1 \cdot \vc_j 
  = |I_a| a + (\tfrac{n}{2}-|I_a|)b.
\]
Equate this expression to \eqref{Eq:jinI2} to show that $|I_a| = \frac{n(b+1)}{2(b-a)}$, which must be an integer.
This proves the claim.
\end{proof}

\begin{remark}\label{Rem:SRG2}
\mbox{}
\begin{enumerate}[$(i)$]
\item 
Suppose that $H$ is a $\BSHM(n,\ell,a,b)$ with $b \ne -a$. 
We may apply the switching transformation \eqref{Eq:switching} if necessary so that $\ell \le \frac{n}{2}$. In fact, by \cref{Prop:SRG2}, we may apply \eqref{Eq:switching} so that
either $\ell < \frac{n}{2}$ and $H$ has Type 1, or else $\ell \le \frac{n}{2}$ and $H$ has Type~2.

\item 
The graph $G$ does not change under the switching transformation \eqref{Eq:switching} (and so its parameters do not change), even though the graph parameters for Type~2 take a different form from those for Type~1.

\item 
The equality in \cref{Prop:SRG2}~(A2),(B2) and the expression for $k_a$ in \cref{Prop:SRG2}~(A4),(B4) 
were also derived in the context of two-distance finite unit-norm tight frames \cite[Theorem 2.4]{BGOY}.
 
\item 
\cref{Prop:SRG2}~(A3),(B3), together with \cref{Prop:SRG1}~(2), recovers \cref{Res:NonEqui}~$(i)$.
\cref{Prop:SRG2} (A2),(B2) strengthens \cref{Res:NonEqui}~$(iii)$ by removing the condition that $\ell>\max\{a,b\}$.

\end{enumerate}
\end{remark}

\end{subsection}

\begin{subsection}{Further Parameter Relations}\label{Subsec:further}

We now derive further relations among the four parameters $n, \ell, a, b$ of a balanced splittable Hadamard matrix (both when $b=-a$ and when $b \ne -a$). 
Recall the notation $n_\vu$ for the number of $-1$ entries in a vector $\vu$ over~$\{1,-1\}$. 

\begin{proposition}\label{Prop:abnotell}
Suppose there exists a nontrivial $\BSHM(n,\ell,a,b)$. 
Then:
\begin{enumerate}[$(i)$]
\item
$\ell \equiv a \equiv b \pmod{2}$

\item
$\ell \equiv a \pmod{4}$ or $\ell \equiv b \pmod{4}$.
\end{enumerate}
\end{proposition}

\begin{proof}
Let $H$ be a nontrivial $\BSHM(n,\ell,a,b)$ with respect to a submatrix~$H_1$,
and let the $i^\text{th}$ column of $H_1$ be~$\vc_i$.
\begin{enumerate}[$(i)$]
\item
By \cref{Lemma:mod4}, we have $\vc_i \cdot \vc_j \equiv \ell \pmod{2}$ for all distinct $i,j$.
By \cref{Res:KS}, both $a$ and~$b$ are attained as dot products of distinct columns of~$H_1$ and therefore $\ell \equiv a \equiv b \pmod{2}$.

\item
Since $n > 2$, there are distinct $i,j$ for which $n_{\vc_i} \equiv n_{\vc_j} \pmod{2}$. For this $i,j$, by \cref{Lemma:mod4} we have $\vc_i \cdot \vc_j \equiv \ell\pmod{4}$. Since $\vc_i \cdot \vc_j \in \{a,b\}$, this gives $\ell \equiv a \pmod{4}$ or $\ell \equiv b \pmod{4}$.
\qedhere
\end{enumerate}
\end{proof}

Recall the quantity $k_a(i)$ introduced in \cref{Defn:kai}, and the result of \cref{Prop:SRG2} that $k_a(i) = k_a$ is independent of~$i$ when $b \ne -a$.
We now use \cref{Prop:abnotell} to determine the value of $k_a$ in terms of $n$ only when $a \not \equiv b \pmod{4}$.

\begin{proposition}\label{Prop:alpha}
Suppose there exists a nontrivial $\BSHM(n,\ell,a,b)$, where $a \not \equiv b \pmod{4}$.
Then $b \ne -a$ and, up to interchange of $a$ and $b$, we have 
$\ell \equiv a \pmod{4}$ and $k_a = \frac{n}{2}-1$.
\end{proposition}

\begin{proof}
Let $H$ be a nontrivial $\BSHM(n,\ell,a,b)$ with respect to a submatrix~$H_1$, where $a \not \equiv b \pmod{4}$.
Then $b \ne -a$ by \cref{Res:b=-a}~$(i)$, and so \cref{Prop:SRG2} applies.

After interchanging $a$ and $b$ if necessary, we have $(a,b) \equiv (\ell,\ell+2) \pmod{4}$ by \cref{Prop:abnotell}~$(i)$.
Then by \cref{Lemma:mod4}, for all distinct $i,j$ we have
\begin{align}
    \vc_i \cdot \vc_j =
    \begin{cases}
      a & \mbox{for $n_{\vc_i} \equiv n_{\vc_j} \pmod{2}$}, \\
      b & \mbox{for $n_{\vc_i} \not \equiv n_{\vc_j} \pmod{2}$}.
   \end{cases} \label{Eq:DotProd}
\end{align}
Since $b$ is attained as a dot product of distinct columns of $H_1$, we may relabel the columns of $H$ if necessary so that $n_{\vc_1} \equiv 0 \pmod{2}$ and $n_{\vc_2} \equiv 1 \pmod{2}$.
Then from \cref{Defn:kai} and \eqref{Eq:DotProd} we obtain
\begin{align*}
    k_a(1) & = \big| \{ t \ne 1: n_{\vc_t} \equiv 0 \pmod{2} \}\big|, \\
    k_a(2) & = \big| \{ t \ne 2: n_{\vc_t} \equiv 1 \pmod{2} \}\big|,
\end{align*}
and so $k_a(1) + k_a(2) = n-2$.
Since $k_a(1) = k_a(2) = k_a$ by \cref{Prop:SRG2}, we conclude that
$k_a=\frac{n}{2}-1$.
\end{proof}

\begin{theorem}\label{Thm:mod4}
Suppose $H$ is a nontrivial $\BSHM(n,\ell,a,b)$, where $a \not \equiv b \pmod{4}$. 
Then, up to application of the switching transformation \eqref{Eq:switching} 
and interchange of $a$ and~$b$, 
the matrix $H$ is a Type~2 $\BSHM(n,2,2,0)$.
\end{theorem}

\begin{proof}
By \cref{Prop:alpha}, we have $b \ne -a$ (and so \cref{Prop:SRG2} applies) and we may interchange $a$ and~$b$ if necessary so that $\ell \equiv a \pmod{4}$ and $k_a = \frac{n}{2}-1$.
Apply the switching transformation \eqref{Eq:switching} if necessary so that either $\ell < \frac{n}{2}$ and $H$ has Type 1, or else $\ell \le \frac{n}{2}$ and $H$ has Type~2.

Suppose firstly that $\ell \le \frac{n}{2}$ and $H$ has Type~2.
Combine $k_a=\frac{n}{2}-1$ and $k_a = \frac{\ell-b}{b-a} + \frac{n(b-1)}{b-a}$ from \cref{Prop:SRG2}~(B4) to obtain
\begin{align}\label{Eq:Mod2n}
    n-\ell+a= \frac{n}{2}\left(a+b\right).
\end{align}
Since $0 \le n-\ell+a \le n$ by \cref{Lemma:A1A2}~$(iii)$, we find that $a+b \in \{0,1,2\}$. Since $a + b \ne 1$ by \cref{Prop:abnotell}~$(i)$, and $b \ne -a$, this implies that $a+b = 2$. Then $\ell=a$ from \eqref{Eq:Mod2n}, so \cref{Prop:SRG2}~(B2) gives $b=0$, and therefore $a=\ell=2$. 
The resulting parameter values are $(\ell,a,b) = (2,2,0)$, and the condition $\ell \le \frac{n}{2}$ is satisfied because $H$ is nontrivial.

Otherwise $\ell < \frac{n}{2}$ and $H$ has Type~1.
By applying the transformation \eqref{Eq:switching} to the above analysis, we find $\ell=n-2$. This does not satisfy the condition $\ell < \frac{n}{2}$ because $H$ is nontrivial.
\end{proof}

We can now characterize the parameters of a nontrivial $\BSHM(n,2,a,b)$, in
\cref{Cor:ell2}~$(i)$. We restate in \cref{Cor:ell2}~$(ii)$ the associated \cref{Res:Constructions}~$(v)$ (derived from \cite[Theorem~3.4]{KS}) and include a short proof.

\begin{corollary}\label{Cor:ell2}
\mbox{}
\begin{enumerate}[$(i)$]
\item
Suppose $H$ is a nontrivial $\BSHM(n,2,a,b)$.
Then, up to application of the switching transformation \eqref{Eq:switching} and interchange of $a$ and~$b$, the matrix $H$
is a Type 2 $\BSHM(n,2,2,0)$.

\item 
A Hadamard matrix of order $n \ge 4$ is equivalent to a Type 2 $\BSHM(n,2,2,0)$.

\end{enumerate}
\end{corollary}

\begin{proof}
\mbox{}
\begin{enumerate}[$(i)$]
\item
We have $a \equiv b \equiv 0 \pmod{2}$ by \cref{Prop:abnotell}~$(i)$.\!
Since $|a|, |b| \le 2$ by \cref{Lemma:A1A2}~$(iii)$, this implies
$a, b \in \{-2,0,2\}$. 
Now $b \ne a$ by \cref{Res:KS}, and $b \ne -a$ by \cref{Prop:SRG1}~(1).
Therefore $a \not \equiv b \pmod{4}$, and the result follows from \cref{Thm:mod4}.

\item
Suppose $H$ is a Hadamard matrix of order $n \ge 4$. Transform $H$, by negating its columns as necessary, to a Hadamard matrix $H'$ whose first row is the all-ones row. Then $H'$ is a $\BSHM(n,2,2,0)$ with respect to its upper $2 \times n$ submatrix.
\qedhere
\end{enumerate}
\end{proof}

In view of \cref{Cor:ell2}, from now on we shall consider a $\BSHM(n,\ell,a,b)$ only in the cases $2 < \ell < n-2$.

\begin{corollary}\label{Cor:mod4}
Suppose there exists a $\BSHM(n,\ell,a,b)$, where $2<\ell<n-2$. Then $\ell \equiv a \equiv b \pmod{4}$.
\end{corollary}
\begin{proof}
By \cref{Prop:abnotell}~$(ii)$, we have $\ell \equiv a \pmod{4}$ or $\ell \equiv b \pmod{4}$.
Since $H$ is nontrivial, by \cref{Thm:mod4} we have $a \equiv b \pmod{4}$ (otherwise $\ell \in \{2,n-2\}$). 
Therefore $\ell \equiv a \equiv b \pmod{4}$.
\end{proof}

\begin{remark}
The nonexistence results of \cref{Res:36}, found by computer search, all occur as special cases of \cref{Cor:mod4}.
\end{remark}

\end{subsection}

\begin{subsection}{Primitive and Imprimitive}

A strongly regular graph $G$ is \emph{primitive} if both $G$ and its complement $\overline{G}$ are connected; otherwise $G$ is \emph{imprimitive}. 
We write $tK_m$ for the union of $t$ disjoint copies of the complete graph~$K_m$. We then have the following characterization of imprimitive strongly regular graphs.

\begin{lemma}[{\cite[p.~117]{BH}}]\label{Lemma:imprimSRG}
Suppose $G$ is a $(v,k,\lambda,\mu)$ strongly regular graph, and let 
$(\overline{k},\,\overline{\lambda},\,\overline{\mu})$ be as in~\eqref{eqn:oGparams}.
Then $G$ is imprimitive if and only if $\mu=0$ or $\overline{\mu} = 0$.
If $\mu = 0$, then $G = tK_m$ 
where $(v,k,\lambda,\mu) = (tm,\,m-1,\,m-2,\,0)$.
If $\overline{\mu} = 0$, then $\overline{G} = tK_m$ where 
$(v,\,\overline{k},\,\overline{\lambda},\,\overline{\mu}) = (tm,\,m-1,\,m-2,\,0)$.
\end{lemma}

The association of a strongly regular graph to a balanced splittable Hadamard matrix given in \cref{Def:assgraph} motivates the following definition.

\begin{definition}\label{Def:primBSHM}
A balanced splittable Hadamard matrix is \emph{primitive} or \emph{imprimitive} according to whether its associated graph $G$ is primitive or imprimitive, respectively.
\end{definition}

In the case $b=-a$, we may take a nontrivial $\BSHM(n,\ell,a,b)$ to be primitive.

\begin{proposition}\label{Prop:imprimitiveb=-a}
Suppose $H$ is a nontrivial $\BSHM(n,\ell,a,-a)$. Then, up to negation of columns, $H$ is primitive.
\end{proposition}

\begin{proof}
Negate columns of $H$ as necessary so that its associated graph is strongly regular with parameters $(v,k,\lambda,\mu)$ as given in \cref{Prop:SRG1}~(3) (or alternatively as in \cref{Prop:SRG1}~(4)).
Suppose, for a contradiction, that $H$ is imprimitive.
By interchanging $a$ and $-a$ (which interchanges $G$ and $\overline{G}$) if necessary, by \cref{Lemma:imprimSRG} we may take $\mu = 0$. Then $a=1$ from \cref{Prop:SRG1}~(3) (or alternatively $a=-1$ from \cref{Prop:SRG1}~(4)), which contradicts that $a$ is even from \cref{Prop:SRG1}~(2).
\end{proof}

In the case $b \ne -a$, we use \cref{Prop:SRG2} to greatly restrict the possible parameters of an imprimitive $\BSHM(n,\ell,a,b)$.
We shall study these parameter sets further in \cref{Sec:imprimitive}.

\begin{proposition}\label{Prop:imprimitivebne-a}
Suppose $H$ is an imprimitive $\BSHM(n,\ell,a,b)$, where $2 < \ell < n-2$ and $b \ne -a$.
Let $G$ be the graph associated with $H$. 
Then, up to application of the switching transformation \eqref{Eq:switching} and interchange of $a$ and~$b$, one of the following holds: 
\begin{enumerate}[$(i)$]
    \item $H$ is a Type 1 $\BSHM(4rs,4s-1,4s-1,-1)$ for some integers $r \ge 2$, $s \ge 1$, and $G = 4sK_r$

    \item $H$ is a Type 2 $\BSHM(8rs,4s,4s,0)$ for some integers $r, s \ge 1$, and $G = 4s K_{2r}$.
\end{enumerate}
\end{proposition} 

\begin{proof}
By \cref{Prop:SRG2}, the graph $G$ is strongly regular and its parameters $(v,k,\lambda,\mu)$ are given by (A4) or~(B4).
Up to interchange of $a$ and $b$ (which interchanges $G$ and $\overline{G}$), by
\cref{Lemma:imprimSRG} we may take $\mu = 0$.
Apply the switching transformation \eqref{Eq:switching} if necessary so that either $\ell < \frac{n}{2}$ and $H$ has Type~1 (in \cref{Prop:SRG2}), or else $\ell \le \frac{n}{2}$ and $H$ has Type~2.  

Suppose firstly that $\ell < \frac{n}{2}$ and $H$ has Type~1.
The expression for $\mu$ in (A4) gives $b \in \{0,-1\}$. We cannot have $b = 0$, otherwise $a=\ell-n$ by (A2) and then $\frac{\ell-b}{b-a} = \frac{\ell}{n-\ell}$ is not an integer because $\ell < \frac{n}{2}$, contrary to~(A3).
Therefore $b = -1$, and then $(\ell-a)(n-\ell-1) = 0$ by~(A2). Since $\ell < n-2$, this gives $a = \ell$. 
Then by \cref{Cor:mod4} we may write $\ell = 4s-1$ for some integer $s \ge 1$.
Since $\frac{n}{b-a} = -\frac{n}{4s}$ is an integer by \cref{Prop:SRG2}~(A3) and $\ell < \frac{n}{2}$ by assumption, we may write $n = 4rs$ for some integer~$r \ge 2$.
Comparison of the parameters $v,k$ in (A4) and \cref{Lemma:imprimSRG} then gives $G = 4sK_r$. This satisfies the conditions for~$(i)$.

Otherwise $\ell \le \frac{n}{2}$ and $H$ has Type~2.
The expression for $\mu$ in (B4) gives $b \in \{0,1\}$. 
We cannot have $b = 1$, otherwise $a=\ell-n$ by (B2) and then 
$\frac{n}{b-a} = \frac{n}{n-\ell+1}$ is not an integer because $2 < \ell \le \frac{n}{2}$, contrary to~(B3).
Therefore $b=0$, and $a=\ell$ by (B2). 
By \cref{Cor:mod4}, we may write $\ell = 4s$ for some integer $s \ge 1$.
Since $\frac{n(b-1)}{2(b-a)} = \frac{n}{8s}$ is an integer by \cref{Prop:SRG2}~(B3), we may then write $n=8rs$ for some integer $r \ge 1$, and comparison of (B4) and \cref{Lemma:imprimSRG} gives $G=4sK_{2r}$. This satisfies the conditions for~$(ii)$.
\end{proof}

\end{subsection}

\begin{subsection}{\cref{Tab:summary} classification}\label{Subsec:classification}

We now combine the results of this section to restrict a $\BSHM(n,\ell,a,b)$ satisfying $2 < \ell < n-2$ to lie in one of five classes.

\begin{theorem}\label{Thm:summary}
Suppose $H$ is a $\BSHM(n,\ell,a,b)$, where $2 < \ell < n-2$, and let $G$ be the graph associated with~$H$.
Then, up to application of the switching transformation \eqref{Eq:switching} and interchange of $a,b$ and (for $b=-a$) negation of columns, one of the five cases displayed in the columns of \cref{Tab:summary} holds.
\end{theorem}

\begin{proof}
Distinguish the cases $b=-a$ and $b \ne -a$.
\begin{description}
\item[Case $\boldsymbol{b = -a}$]
Apply the switching transformation \eqref{Eq:switching} if necessary so that $\ell \le \frac{n}{2}$, and interchange $a$ and $b$ if necessary so that $a > 0$ (noting that $a=0$ is excluded by \cref{Prop:SRG1}~(1)).
Then $a$ is even and $\frac{\ell}{a}$ is an odd integer and $\frac{n}{4a}$ is an integer, by \cref{Prop:SRG1}~(2).
This implies that $\ell \equiv a \pmod{4}$ and $\ell \ne \frac{n}{2}$.
\cref{Prop:SRG1}~(1) shows that $\ell \ne a^2$, and so
$n = \frac{\ell^2-a^2}{\ell-a^2}$.
\cref{Prop:imprimitiveb=-a} shows that, up to negation of columns, $H$ is primitive.
This establishes the conditions for the first column of \cref{Tab:summary}.

\item[Case $\boldsymbol{b \ne -a}$]
Distinguish the cases $H$ is imprimitive and $H$ is primitive.

\begin{description}
\item[Case $H$ is imprimitive]
\cref{Prop:imprimitivebne-a} shows that the conditions are satisfied for the second and fourth columns of \cref{Tab:summary}. 

\item[Case $H$ is primitive]
Since $ab \le 0$ by \cref{Prop:SRG2}~(A2) and (B2), we may interchange $a$ and~$b$ if necessary so that $a > 0 \ge b$ (noting that $a = b =0$ is excluded by \cref{Res:KS}).
Apply the switching transformation \eqref{Eq:switching} if necessary so that either $\ell < \frac{n}{2}$ and $H$ has Type 1, or else $\ell \le \frac{n}{2}$ and $H$ has Type~2.

\begin{description}
\item[Case $\ell < \frac{n}{2}$ and $H$ has Type 1]
\cref{Prop:SRG2}~(A2) shows that $\ell +ab \ne 0$, otherwise $(a,b) = (\ell,-1)$ and then $\mu=0$ by \cref{Prop:SRG2}~(A4) and so $H$ is imprimitive.
Therefore
$n = \frac{(\ell-a)(\ell-b)}{\ell+ab}$ from \cref{Prop:SRG2}~(A2), and
$\frac{\ell-b}{b-a}$ and $\frac{n}{b-a}$ and $\frac{n(b+1)}{2(b-a)}$ and $\frac{nb(b+1)}{(b-a)^2}$ are integers from \cref{Prop:SRG2}~(A3).
\cref{Cor:mod4} shows that $\ell \equiv a \equiv b \pmod{4}$.
This establishes the conditions for the third column of \cref{Tab:summary}.

\item[Case $\ell \le \frac{n}{2}$ and $H$ has Type~2]
\cref{Prop:SRG2}~(B2) shows that $\ell +ab -a-b\ne 0$, otherwise $(a,b) = (\ell,0)$ and then $\mu=0$ by \cref{Prop:SRG2}~(B4) and so $H$ is imprimitive.
Therefore
$n = \frac{(\ell-a)(\ell-b)}{\ell+ab-a-b}$ from \cref{Prop:SRG2}~(B2), and
$\frac{\ell-b}{b-a}$ and $\frac{n}{b-a}$ and $\frac{n(b-1)}{2(b-a)}$ and $\frac{nb(b-1)}{(b-a)^2}$ are integers from \cref{Prop:SRG2}~(B3).
\cref{Cor:mod4} shows that $\ell \equiv a \equiv b \pmod{4}$.
This establishes the conditions for the fifth column of \cref{Tab:summary}.
\qedhere
\end{description}

\end{description}

\end{description}
\end{proof}

\end{subsection}

\end{section}

\begin{section}{Open primitive cases}\label{Sec:primitive-open}
In this section, we tabulate parameter sets $(n,\ell,a,b)$ for small $n$ for which the existence of a nontrivial primitive $\BSHM(n,\ell,a,b)$ is not determined by the results presented so far. We indicate which of these open cases will be settled by the results of \cref{Sec:primitive}.

 From \cref{Tab:summary}, there are three possibilities for a nontrivial primitive $\BSHM(n,\ell,a,b)$ $H$: the case $b=-a$ (first column), the case $b\ne-a$ where $H$ has Type~1 (third column), and the case $b\ne-a$ where $H$ has Type~2 (the fifth column).
In each of these cases, there is at most one possible value of $n$ for each triple $(\ell,a,b)$; this contrasts with the imprimitive cases in which infinitely many values of $n$ are possible.

The following procedures identify all parameter sets for which the existence of a nontrivial primitive $\BSHM(n,\ell,a,b)$ $H$ is not ruled out by the constraints presented in \cref{Tab:summary}.

\begin{description}
\item[Case $\boldsymbol{b=-a}$]
Fix an even value of $\ell > 1$.
For each $a$ for which $0 < a < \sqrt{\ell}$ and $a \equiv \ell \pmod{4}$
and $\frac{\ell}{a}$ is an odd integer,
calculate $n = \frac{\ell^2-a^2}{\ell-a^2}$. Retain the parameter values $(n,\ell,a,-a)$ provided that $\frac{n}{4a}$ is an integer and $n > 2 \ell$.

The 16 parameter sets $(n,\ell,a,-a)$ with $n \le 1296$ retained under this procedure are displayed in \cref{Tab:surviveb=-a}. 
Existence for five of these parameter sets is given by \cref{Res:Equi}. 
Nonexistence for the parameter set $(96,20,4,-4)$ is given by the nonexistence of the associated strongly regular graph in \cref{Prop:SRG1}~(3) (as well as that in \cref{Prop:SRG1}~(4))~\cite{Brouwer}.
Existence for the remaining ten parameter sets remains open. 

\begin{table}[ht!]
\caption{Parameter sets $(n,\ell,a,-a)$ satisfying the constraints in \cref{Tab:summary} for a nontrivial $\BSHM(n,\ell,a,-a)$ with $n \le 1296$,
and the parameter set $(v,k,\lambda,\mu)$ of the associated strongly regular graph under each of the transformations of \cref{Prop:SRG1}~(3) and~(4).}
\begin{center}
\small
\begin{tabular}{|c|c|c|c|c|} 		
\hline
$(n,\ell,a,-a)$		& $(v,k,\lambda,\mu)$ 	& $(v,k,\lambda,\mu)$ 	& BSHM exists? 	& reason		\\ 
			& under Prop.~\ref{Prop:SRG1}~(3)			
						& under Prop.~\ref{Prop:SRG1}~(4)	
									&		&			\\ \hline
$(16,6,2,-2)$        	& $(16,6,2,2)$		& $(16,10,6,6)$		& yes		& \cref{Res:Equi}	\\ 
$(64,28,4,-4)$       	& $(64,28,12,12)$	& $(64,36,20,20)$	& yes		& \cref{Res:Equi}	\\ 
$(96,20,4,-4)$       	& $(96,45,24,18)$ 	& $(96,57,36,30)$	& no		& \cite{Brouwer} 	\\ 
$(144,66,6,-6)$      	& $(144,66,30,30)$	& $(144,78,42,42)$	& open  	&			\\ 
$(256,120,8,-8)$     	& $(256,120,56,56)$	& $(256,136,72,72)$	& yes 		& \cref{Res:Equi}  	\\ 
$(288,42,6,-6)$      	& $(288,140,76,60)$	& $(288,164,100,84)$	& open  	&			\\ 
$(320,88,8,-8)$      	& $(320,154,78,70)$	& $(320,174,98,90)$	& open  	&			\\ 
$(400,190,10,-10)$   	& $(400,190,90,90)$	& $(400,210,110,110)$	& open  	&			\\ 
$(560,130,10,-10)$   	& $(560,273,140,126)$	& $(560,301,168,154)$	& open  	&			\\ 
$(576,276,12,-12)$   	& $(576,276,132,132)$	& $(576,300,156,156)$	& yes		& \cref{Res:Equi}  	\\ 
$(640,72,8,-8)$      	& $(640,315,170,140)$	& $(640,355,210,180)$	& open  	&			\\ 
$(784,378,14,-14)$   	& $(784,378,182,182)$	& $(784,406,210,210)$	& open  	&			\\ 
$(1008,266,14,-14)$  	& $(1008,494,250,234)$	& $(1008,530,286,270)$	& open  	&			\\ 
$(1024,496,16,-16)$  	& $(1024,496,240,240)$	& $(1024,528,272,272)$	& yes		& \cref{Res:Equi}  	\\ 
$(1200,110,10,-10)$  	& $(1200,594,318,270)$	& $(1200,654,378,330)$	& open  	&			\\ 
$(1296,630,18,-18)$  	& $(1296,630,306,306)$	& $(1296,666,342,342)$	& open  	&			\\ \hline
\end{tabular}
\end{center}

\label{Tab:surviveb=-a}
\end{table}

\item[Case $\boldsymbol{b \ne -a}$ where $H$ has Type 1]
Fix a value of $\ell > 2$.
For each $a, b$ for which $\ell \ge a > 0 \ge b \ge -\ell$ and $b \ne -a$ and $a \equiv b \equiv \ell \pmod{4}$
 and $ab > -\ell$ and $\frac{\ell-b}{b-a}$ is an integer,
calculate $n = \frac{(\ell-a)(\ell-b)}{\ell+ab}$.
Retain the parameter values $(n,\ell,a,b)$ provided that $\frac{n}{b-a}$ and $\frac{n(b+1)}{2(b-a)}$ and $\frac{nb(b+1)}{(b-a)^2}$ are integers and $n > 2\ell$,
and provided that the existence of the associated strongly regular graph $G$ is not ruled out in Brouwer's tables of parameters of strongly regular graphs~\cite{Brouwer},
and provided that the parameters $\overline{\lambda},\overline{\mu}$ given by \eqref{eqn:oGparams} for the complementary graph are both non-negative.

The 30 parameter sets $(n,\ell,a,b)$ with $n \le 256$ obtained via this procedure are displayed in \cref{Tab:survivebne-aType1}.
Nonexistence for the parameter set $(96,19,3,-5)$ is given by the nonexistence of the associated strongly regular graph.
Existence for 16 of the 29 remaining parameter sets will be demonstrated in \cref{Sec:primitive}.

\begin{table}[ht!]
\caption{Parameter sets $(n,\ell,a,b)$ satisfying the constraints in \cref{Tab:summary} for a nontrivial primitive $\BSHM(n,\ell,a,b)$ of Type 1 with $b \ne -a$ and $n \le 256$, and the parameters $(v,k,\lambda,\mu)$ of the associated strongly regular graph.}
\begin{center}
\normalsize
\begin{tabular}{|c|c|c|c|} 		
\hline
$(n,\ell,a,b)$		& $(v,k,\lambda,\mu)$ 	& BSHM exists? 	& reason			\\ \hline
$(16,5,1,-3)$  		& $(16,10,6,6)$		& yes		& \cref{Cor:Had-add-subtract} 	\\ \hline
$(64,14,6,-2)$ 		& $(64,14,6,2)$		& yes		& \cref{Cor:PDS}~$(i)$  	\\
$(64,18,2,-6)$ 		& $(64,45,32,30)$	& yes		& \cref{Cor:PDS}~$(ii)$		\\
$(64,21,5,-3)$ 		& $(64,21,8,6)$		& yes		& \cref{Cor:PDS}~$(i)$  	\\
$(64,27,3,-5)$ 		& $(64,36,20,20)$	& yes		& \cref{Cor:Had-add-subtract} 	\\ \hline
$(96,19,3,-5)$ 		& $(96,57,36,30)$ 	& no		& \cite{Brouwer}		\\
$(96,38,2,-10)$		& $(96,76,60,60)$	& open		&		  		\\
$(96,45,9,-3)$ 		& $(96,20,4,4)$		& open		&				\\ \hline
$(120,51,3,-9)$		& $(120,85,60,60)$	& open		&				\\
$(120,56,8,-4)$		& $(120,35,10,10)$	& open		&				\\ \hline 
$(144,22,10,-2)$	& $(144,22,10,2)$	& open		&		  		\\ 
$(144,33,9,-3)$		& $(144,33,12,6)$	& open		&		  		\\ 
$(144,39,3,-9)$		& $(144,104,76,72)$	& open		&		  		\\ 
$(144,44,8,-4)$		& $(144,44,16,12)$	& open		&		  		\\
$(144,52,4,-8)$		& $(144,91,58,56)$	& open		&		  		\\
$(144,55,7,-5)$		& $(144,55,22,20)$	& open		&		  		\\
$(144,65,5,-7)$		& $(144,78,42,42)$	& open		&		  		\\ \hline
$(216,40,4,-8)$		& $(216,140,94,84)$	& open		& 				\\ 
$(216,43,7,-5)$		& $(216,86,40,30)$	& open		& 				\\ \hline       
$(256,30,14,-2)$	& $(256,30,14,2)$	& yes		& \cref{Cor:PDS}~$(i)$		\\
$(256,45,13,-3)$	& $(256,45,16,6)$	& yes		& \cref{Cor:PDS}~$(i)$		\\
$(256,51,3,-13)$	& $(256,204,164,156)$	& yes		& \cref{Cor:PDS}~$(v)$		\\
$(256,60,12,-4)$	& $(256,60,20,12)$	& yes		& \cref{Cor:PDS}~$(i)$		\\
$(256,68,4,-12)$	& $(256,187,138,132)$	& yes		& \cref{Cor:PDS}~$(iii)$	\\
$(256,75,11,-5)$	& $(256,75,26,20)$	& yes		& \cref{Cor:PDS}~$(i)$		\\
$(256,85,5,-11)$	& $(256,170,114,110)$	& yes		& \cref{Cor:PDS}~$(v)$		\\
$(256,90,10,-6)$	& $(256,90,34,30)$	& yes		& \cref{Cor:PDS}~$(i)$		\\
$(256,102,6,-10)$	& $(256,153,92,90)$	& yes		& \cref{Cor:PDS}~$(v)$		\\
$(256,105,9,-7)$	& $(256,105,44,42)$	& yes		& \cref{Cor:PDS}~$(i)$		\\
$(256,119,7,-9)$	& $(256,136,72,72)$	& yes		& \cref{Cor:Had-add-subtract}	\\ \hline
\end{tabular}
\end{center}
\normalsize
\label{Tab:survivebne-aType1}
\end{table}

\item[Case $\boldsymbol{b \ne -a}$ where $H$ has Type 2]
Fix a value of $\ell > 2$.
For each $a, b$ for which $\ell \ge a > 0 \ge b \ge -\ell$ and $b \ne -a$ and $a \equiv b \equiv \ell \pmod{4}$ and $ab-a-b > -\ell$ and $\frac{\ell-b}{b-a}$ is an integer,
calculate $n = \frac{(\ell-a)(\ell-b)}{\ell+ab-a-b}$.
Retain the parameter values $(n,\ell,a,b)$ provided that $\frac{n}{b-a}$ and $\frac{n(b-1)}{2(b-a)}$ and $\frac{nb(b-1)}{(b-a)^2}$ are integers and $n \ge 2\ell$,
and provided that the existence of the associated strongly regular graph $G$ is not ruled out by~\cite{Brouwer},
and provided that the parameters $\overline{\lambda},\overline{\mu}$ given by \eqref{eqn:oGparams} for the complementary graph are both non-negative.

The 30 parameter sets $(n,\ell,a,b)$ with $n \le 256$ obtained via this procedure are displayed in \cref{Tab:survivebne-aType2}. The apparent relationship between the entries in \cref{Tab:survivebne-aType1,Tab:survivebne-aType2} will be explained in \cref{Prop:add-subtract}.
Nonexistence for the parameter set $(96,21,5,-3)$ is given by the nonexistence of the associated strongly regular graph.
Existence for 16 of the 29 remaining parameter sets will be demonstrated in \cref{Sec:primitive}.

\begin{table}[ht!]
\caption{Parameter sets $(n,\ell,a,b)$ satisfying the constraints in \cref{Tab:summary} for a nontrivial primitive $\BSHM(n,\ell,a,b)$ of Type 2 with $b \ne -a$ and $n \le 256$, and the parameters $(v,k,\lambda,\mu)$ of the associated strongly regular graph.}
\begin{center}
\begin{tabular}{|c|c|c|c|} 		
\hline
$(n,\ell,a,b)$		& $(v,k,\lambda,\mu)$ 	& BSHM exists? 	& reason			\\ \hline
$(16,7,3,-1)$  		& $(16,6,2,2)$		& yes		& \cref{Cor:Had-add-subtract}  	\\ \hline
$(64,15,7,-1)$ 		& $(64,14,6,2)$		& yes		& \cref{Cor:PDS}$(i)$  		\\
$(64,19,3,-5)$ 		& $(64,45,32,30)$	& yes		& \cref{Cor:PDS}$(ii)$ 		\\
$(64,22,6,-2)$ 		& $(64,21,8,6)$		& yes		& \cref{Cor:PDS}$(i)$  		\\
$(64,29,5,-3)$ 		& $(64,28,12,12)$	& yes		& \cref{Cor:Had-add-subtract}  	\\ \hline
$(96,21,5,-3)$		& $(96,45,24,18)$ 	& no		& \cite{Brouwer} 		\\ 
$(96,39,3,-9)$ 		& $(96,76,60,60)$		& open		&				\\ 
$(96,46,10,-2)$		& $(96,20,4,4)$		& open		&				\\ \hline
$(120,52,4,-8)$		& $(120,85,60,60)$	& open		&				\\ 
$(120,57,9,-3)$		& $(120,35,10,10)$	& open		&				\\ \hline
$(144,23,11,-1)$	& $(144,22,10,2)$	& open		&				\\
$(144,34,10,-2)$	& $(144,33,12,6)$	& open		&				\\
$(144,40,4,-8)$		& $(144,104,76,72)$	& open		&				\\
$(144,45,9,-3)$		& $(144,44,16,12)$	& open		&				\\
$(144,53,5,-7)$		& $(144,91,58,56)$	& open		&				\\
$(144,56,8,-4)$		& $(144,55,22,20)$	& open		&				\\
$(144,67,7,-5)$		& $(144,66,30,30)$	& open		&				\\ \hline
$(216,41,5,-7)$		& $(216,140,94,84)$	& open		& 				\\
$(216,44,8,-4)$		& $(216,86,40,30)$	& open		& 				\\ \hline
$(256,31,15,-1)$	& $(256,30,14,2)$	& yes		& \cref{Cor:PDS}$(i)$		\\
$(256,46,14,-2)$	& $(256,45,16,6)$	& yes		& \cref{Cor:PDS}$(i)$		\\
$(256,52,4,-12)$	& $(256,204,164,156)$	& yes		& \cref{Cor:PDS}$(v)$		\\
$(256,61,13,-3)$	& $(256,60,20,12)$	& yes		& \cref{Cor:PDS}$(i)$		\\
$(256,69,5,-11)$	& $(256,187,138,132)$	& yes		& \cref{Cor:PDS}$(iii)$		\\
$(256,76,12,-4)$	& $(256,75,26,20)$	& yes		& \cref{Cor:PDS}$(i)$		\\
$(256,86,6,-10)$	& $(256,170,114,110)$	& yes		& \cref{Cor:PDS}$(v)$		\\
$(256,91,11,-5)$	& $(256,90,34,30)$	& yes		& \cref{Cor:PDS}$(i)$		\\
$(256,103,7,-9)$	& $(256,153,92,90)$	& yes		& \cref{Cor:PDS}$(v)$		\\
$(256,106,10,-6)$	& $(256,105,44,42)$	& yes		& \cref{Cor:PDS}$(i)$		\\
$(256,121,9,-7)$	& $(256,120,56,56)$	& yes		& \cref{Cor:Had-add-subtract}	\\ \hline
\end{tabular}
\end{center}
\label{Tab:survivebne-aType2}
\end{table}

\end{description}

\end{section}

\begin{section}{Constructions for the primitive case}\label{Sec:primitive}
In this section, we develop the study of primitive balanced splittable Hadamard matrices.
In \cref{Subsec:all-ones}, we show that we can regard a balanced splittable Hadamard matrix $H$ as having both parameter sets $(n,\ell,a,b)$ and $(n,\ell+1,a+1,b+1)$, provided that $H$ contains the all-ones row. 
In \cref{Subsec:pds}, we use partial difference sets in elementary abelian $2$-groups to construct new infinite families of primitive balanced splittable Hadamard matrices. 

\begin{subsection}{Inclusion or exclusion of the all-ones row}\label{Subsec:all-ones}

Inspection of the open parameter sets $(n,\ell,a,b)$ and $(v,k,\lambda,\mu)$ appearing in \cref{Tab:surviveb=-a,Tab:survivebne-aType1,Tab:survivebne-aType2} reveals some relationships.
Each parameter set $(n,\ell,a,-a)$ in \cref{Tab:surviveb=-a} has two associated 
strongly regular graph parameter sets $(v,k_1,\lambda_1,\mu_1)$ and $(v,k_2,\lambda_2,\mu_2)$. 
For $n \le 256$, there are 
corresponding parameter sets $(n,\ell-1,a-1,-a-1)$ and $(v,k_2,\lambda_2,\mu_2)$ in \cref{Tab:survivebne-aType1}, and
corresponding parameter sets $(n,\ell+1,a+1,-a+1)$ and $(v,k_1,\lambda_1,\mu_1)$ in \cref{Tab:survivebne-aType2}.
Furthermore, all other entries $(n,\ell,a,b)$ and $(v,k,\lambda,\mu)$ in \cref{Tab:survivebne-aType1} are in one-to-one correspondence with entries $(n,\ell+1,a+1,b+1)$ and $(v,k,\lambda,\mu)$ in \cref{Tab:survivebne-aType2}.
These apparent relationships are all explained by \cref{Prop:add-subtract}, based on the following remark.

\begin{remark}\label{Rem:addrow}
Let $H = \begin{pmatrix} H_1 \\ \vo^T \\ H_2 \end{pmatrix}$, where $H_1$ and $H_2$ each contain at least one row. 
Then the following are equivalent:
\begin{enumerate}[$(i)$]
\item $H$ is a $\BSHM(n,\ell,a,b)$ with respect to $H_1$ and has associated graph~$G$
\item $H$ is a $\BSHM(n,\ell+1,a+1,b+1)$ with respect to $\begin{pmatrix} H_1 \\ \vo^T \end{pmatrix}$ and has associated graph~$G$.
\end{enumerate}
\end{remark}

\begin{proposition} \label{Prop:add-subtract}
\mbox{}
\begin{enumerate}[$(i)$]
\item
Suppose there exists a $\BSHM(n,\ell,a,-a)$, where $2 < \ell < n-2$.
Then there exists a Type 1 $\BSHM(n,\ell-1,a-1,-a-1)$ containing the all-ones row and having the same associated graph as some $\BSHM(n,\ell,a,-a)$,
and there exists a Type 2 $\BSHM(n,\ell+1,a+1,-a+1)$ containing the all-ones row and having the same associated graph as some $\BSHM(n,\ell,a,-a)$.

\item
Let $1 < \ell < n-2$ and $b \not \in \{-a, -a-2\}$.
Then there exists a Type 1 $\BSHM(n,\ell,a,b)$ containing the all-ones row and having associated graph~$G$
if and only if
there exists a Type 2 $\BSHM(n,\ell+1,a+1,b+1)$ containing the all-ones row and having associated graph~$G$.
\end{enumerate}
\end{proposition}

\begin{proof}
\mbox{}
\begin{enumerate}[$(i)$]
\item
Let $H$ be a $\BSHM(n,\ell,a,-a)$.
By \cref{Prop:SRG1}~(4), $H$ can be transformed to a matrix $H'' = \begin{pmatrix}H''_1 \\ \vo^T \\ H''_2 \end{pmatrix}$ that is a $\BSHM(n,\ell,a,-a)$ with respect to $\begin{pmatrix} H''_1 \\ \vo^T \end{pmatrix}$;
let its associated graph be~$G''$.
Then $H''$ is a $\BSHM(n,\ell-1,a-1,-a-1)$ with respect to $H''_1$ and has associated graph $G''$ by \cref{Rem:addrow}, and $H''$ has Type~1 by \cref{Prop:SRG2}~(A1).

By \cref{Prop:SRG1}~(3), $H$ can instead be transformed to a matrix $H' = \begin{pmatrix}H'_1 \\ \vo^T \\ H'_2 \end{pmatrix}$ that is a $\BSHM(n,\ell,a,-a)$ with respect to~$H'_1$; let its associated graph be~$G'$.
Then $H'$ is a $\BSHM(n,\ell+1,a+1,-a+1)$ with respect to $\begin{pmatrix} H'_1 \\ \vo^T \end{pmatrix}$ and has associated graph $G'$ by \cref{Rem:addrow}, and $H'$ has Type~2 by \cref{Prop:SRG2}~(B1).

\item
Let $H = \begin{pmatrix} H_1 \\ H_2 \end{pmatrix}$ be a Type~1 $\BSHM(n,\ell,a,b)$ with respect to $H_1$ containing the all-ones row~$\vo^T$ and having associated graph~$G$.
By \cref{Prop:SRG2}~(A1), the all-ones row is contained in~$H_2$ and so we can write 
$H= \begin{pmatrix} H_1\\ \vo^T \\ H_2' \end{pmatrix}$.
Then $H$ is a $\BSHM(n,\ell+1,a+1,b+1)$ with respect to $\begin{pmatrix} H_1\\ \vo^T \end{pmatrix}$ and has associated graph~$G$
by \cref{Rem:addrow}, 
and $H$ has Type~2 by \cref{Prop:SRG2}~(B1) (noting that $b+1 \ne -(a+1)$ by assumption).

The converse is similar.
\qedhere
\end{enumerate}
\end{proof}

\begin{remark}
\mbox{}
\begin{enumerate}[$(i)$]
\item
\cref{Prop:add-subtract}~$(i)$ extends \cite[Remark~2.8~(2)]{KS}.

\item
The equivalence described in \cref{Prop:add-subtract}~$(ii)$ relies on the presence of the all-ones row. In general, one cannot assume that a $\BSHM(n,\ell,a,b)$ contains the all-ones row.
\cref{Prop:add-subtract}~$(i)$ describes an exception in the case $b=-a$, when negation of a column of the matrix leaves all column dot products of the upper $\ell \times n$ submatrix in $\{-a,a\}$ (see \cref{Prop:SRG1}). 
\cref{Prop:0to-1} will describe a further exception in the case $(n,\ell,a,b) = (4rs,4s,4s,0)$.
\cref{Rem:skew-type} describes an example of a result that holds in the presence of the all-ones row but otherwise fails.
\end{enumerate}
\end{remark}

Application of \cref{Prop:add-subtract}~$(i)$ to \cref{Res:Equi} gives the following result.
\begin{corollary}\label{Cor:Had-add-subtract}
Suppose $n>1$ is the order of a Hadamard matrix. Then there exists
a Type 1 $\BSHM(4n^2,2n^2-n-1,n-1,-n-1)$ and 
a Type 2 $\BSHM(4n^2,2n^2-n+1,n+1,-n+1)$.
\end{corollary}

\end{subsection}

\begin{subsection}{Constructions from partial difference sets}\label{Subsec:pds}

In this subsection, we use partial difference sets in elementary abelian $2$-groups to construct
several new infinite families of primitive balanced splittable Hadamard matrices.
In particular, we determine for each parameter set with $n \in \{64,256\}$
whether a nontrivial primitive $\BSHM(n,\ell,a,b)$ exists.
We then use packings of partial difference sets to produce infinite families of Hadamard matrices that have the balanced splittable property with respect to multiple disjoint submatrices simultaneously.

We begin by defining a partial difference set and reviewing its basic properties.
For a subset $D$ of a multiplicative group, we write $D^{(-1)}=\{ d^{-1} \mid d \in D \}$. 

\begin{definition}\label{def-pds}
Let $D$ be an $\ell$-subset of a multiplicative group $G$ of order~$v$. The subset $D$ is a $(v,\ell,\alpha,\beta)$ \emph{partial difference set} in $G$ if the multiset $\{xy^{-1} \mid x,y \in D, x \ne y \}$ contains each nonidentity element of $D$ exactly $\alpha$ times and each nonidentity element of $G \setminus D$ exactly $\beta$ times.
The partial difference set is \emph{regular} if $1_G \notin D$ and $D=D^{(-1)}$.
\end{definition}

The parameters of a partial difference set are usually written as $(v,k,\lambda,\mu)$, but we have instead written $(v,\ell,\alpha,\beta)$ to avoid confusion with the strongly regular graph parameters $(v,k,\lambda,\mu)$ used in \cref{Prop:SRG1,Prop:SRG2}.
The condition $D=D^{(-1)}$ in \cref{def-pds} is guaranteed to hold when $\alpha \ne \beta$ \cite[Prop.~1.2]{Ma}.
In that case, the condition $1_G \notin D$ is not restrictive: if $D$ is a partial difference set and $D=D^{(-1)}$ and $1_G \in D$, then $D \setminus \{1_G\}$ is a regular partial difference set \cite[p.~222]{Ma}.

Let $G$ be a finite abelian group. 
The \emph{exponent} $\exp(G)$ of $G$ is the smallest positive integer $n$ for which $g^n = 1_G$ for each $g \in G$.
A \emph{character} $\chi$ of $G$ is a group homomorphism from $G$ to the multiplicative group of the complex field~$\C$. 
The image of a character $\chi$ is therefore the multiplicative group of complex $(\exp(G))^\text{th}$ roots of unity.
The set of characters of $G$ forms a group $\wh{G}$ under the operation $(\chi_1 \circ \chi_2)(s) = \chi_1(s) \chi_2(s)$ for all $s \in G$,
and the groups $G$ and $\wh{G}$ are isomorphic.
We may therefore index the elements of $\wh{G}$ as $\{ \chi_g : g \in G\}$, and since there is an isomorphism mapping each $g$ to $\chi_g$ we have $\chi_g \circ \chi_h = \chi_{gh}$ for all $g, h \in G$. Therefore
$\chi_g(s)\chi_h(s) = \chi_{gh}(s)$ for all $g,h,s \in G$.
The \emph{principal character} $\chi_{1_G}$ of $G$ is the identity element of $\wh{G}$, satisfying $\chi_{1_G}(s)=1$ for each $s \in G$.
All other characters of $G$ are \emph{nonprincipal}.
For a subset $D$ of $G$ and $\chi \in \wh{G}$, the \emph{character sum} of $\chi$ on $D$ is $\chi(D)=\sum_{d \in D} \chi(d)$. 
See \cite[Section~2]{JL22+} for an introduction to the use of characters of finite abelian groups, and for example \cite[Chapter~1]{Pott} for a comprehensive treatment.

A partial difference set satisfying $D = D^{(-1)}$ can be characterized in terms of the values of its nonprincipal character sums.

\begin{proposition}[{\cite[Corollary~3.3]{Ma}}]\label{Prop:PDSchar}
Let $G$ be an abelian group of order $v$, and let $D$ be an $\ell$-subset of~$G$ satisfying $D = D^{(-1)}$.
Then the following are equivalent:
\begin{enumerate}[$(i)$]
\item
$D$ is a $(v,\ell,\alpha,\beta)$ partial difference set in $G$ 

\item
$\alpha,\beta$ are nonnegative integers satisfying $\ell^2=\gamma +(\alpha-\beta)\ell+\beta v$, where
\[ 
\gamma = \begin{cases} 
\ell-\beta & \mbox{if $1_G \not \in D$,} \\
\ell-\alpha & \mbox{if $1_G \in D$},
\end{cases}
\]
and
\[
\chi(D) = \tfrac{1}{2}\Big(\alpha-\beta \pm \sqrt{(\alpha-\beta)^2+4\gamma}\Big)
\quad \mbox{for all nonprincipal characters $\chi$ of $G$}.
\]
\end{enumerate}
\end{proposition}

\cref{Prop:PDSchar} implies that the nonprincipal character sums $\chi(D)$ of a $(v,\ell,\alpha,\beta)$ partial difference set $D$ in $G$ satisfying $D = D^{(-1)}$ take the values $a$ and $b$ if and only if 
\begin{equation}\label{Eq:lambda-mu}
(\alpha,\,\beta) = \begin{cases}
(\ell+ab+a+b,\,\ell+ab) 	& \mbox{if $1_G \not \in D$} \\
(\ell+ab,\, \ell+ab-a-b) 	& \mbox{if $1_G \in D$} .
\end{cases}
\end{equation}
See the survey article \cite{Ma} for a detailed treatment of partial difference sets, and \cite{JL21} for some recent constructions.

\begin{definition}\label{Defn:C(D)}
Let $S$ be a subset of a finite abelian group $G$.
The \emph{character table $C(S)$ of $S$ in $G$} is the 
matrix with rows indexed by elements of~$S$ and columns indexed by elements of~$G$,
whose $(s,g)$ entry is $\chi_g(s)$ for $s \in S$ and $g \in G$.
\end{definition}

\begin{lemma} \label{Lemma:C(S)}
Let $S$ be a subset of a finite abelian group $G$, let $C(S)$ be the character table of $S$ in $G$, and let $g, h \in G$.
Then the dot product of the columns of $C(S)$ indexed by $g$ and $h$ is $\chi_{gh}(S)$.
\end{lemma}
\begin{proof}
Let the column of $C(S)$ indexed by $g \in G$ be $\vc_g$.
Then 
\[
\vc_g \cdot \vc_h = 
\sum_{s \in S} \chi_g(s) \chi_h(s) =
\sum_{s \in S} \chi_{gh}(s) =
\chi_{gh}(S).
\]
\end{proof}

We now show how to form a balanced splittable Hadamard matrix from a partial difference set in~$\Z_2^r$, via the character table of~$\Z_2^r$.

\begin{theorem}\label{Thm:PDS}
Let $D$ be a subset of~$\Z_2^r$.
Then 
$C(\Z_2^r)$ contains the all-ones row, 
and
$D=D^{(-1)}$, and 
the following are equivalent:
\begin{enumerate}[$(i)$]
\item
$C(\Z_2^r)$ is a $\BSHM(2^r,\ell,a,b)$ with respect to $C(D)$,
having Type 1 if $1_{\Z_2^r} \not \in D$ and $b \ne -a$ and having Type 2 if $1_{\Z_2^r} \in D$ and $b \ne -a$
\item
$D$ is a partial difference set of cardinality $\ell$ in $\Z_2^r$ whose nonprincipal character sums take the values $a$ and~$b$ (and whose parameters $\alpha,\beta$ are given by \eqref{Eq:lambda-mu} with $G = \Z_2^r$).
\end{enumerate}
\end{theorem}

\begin{proof}
Since $\exp(\Z_2^r) =2$, each character of $\Z_2^r$ takes values in $\{1,-1\}$. 
Therefore $C(\Z_2^r)$ is an order $2^r$ matrix over $\{1,-1\}$, and contains the all-ones row (namely the row indexed by $1_{\Z_2^r}$). 
This row is contained in the $|D| \times 2^r$ submatrix $C(D)$ if and only if $1_{\Z_2^r} \in D$. 

Since $h = h^{-1}$ for each $h \in \Z_2^r$, we have $D = D^{(-1)}$ and we may replace $\chi_{gh}$ by $\chi_{gh^{-1}}$ when applying \cref{Lemma:C(S)} with $G = \Z_2^r$.  
Take $S=\Z_2^r$ in \cref{Lemma:C(S)}. Then the dot product of distinct columns of $C(\Z_2^r)$ indexed by $g$ and $h$ is
$\chi_{gh^{-1}}(\Z_2^r)$, which equals $0$ because $\chi_{gh^{-1}}$ is nonprincipal for $g \ne h$.
Therefore $C(\Z_2^r)$ is a Hadamard matrix.
Now take $S=D$ in \cref{Lemma:C(S)}, and let the column of $C(D)$ indexed by $g \in \Z_2^r$ be $\vc_g$.
Then
\begin{align*}
\{\vc_g \cdot \vc_h : g,h \in \Z_2^r \mbox{ are distinct}\} 
 &= \{\chi_{gh^{-1}}(D) : g,h \in \Z_2^r \mbox{ are distinct}\} \\
 &= \{\chi(D) : \chi \mbox{ is nonprincipal on }\Z_2^r\}.
\end{align*}
This, together with \cref{Prop:PDSchar} and \cref{Prop:SRG2}~(A1),(A2),(B2), 
establishes the equivalence of $(i)$ and~$(ii)$.
\end{proof}

\begin{remark}
\mbox{}
\begin{enumerate}[$(i)$]
\item
In the special case $b=-a$, \cref{Thm:PDS} demonstrates the equivalence of a difference set in $\Z_2^r$ with a real harmonic ETF over $\{1,-1\}$, as has been shown in \cite{DF,XZG}; see also \cite[Section 2.3]{STD+}.

\item
A connection between a partial difference set and a biangular (two-distance) tight frame was noted in \cite[Theorem 4.24]{CFHT}; 
however, the partial difference set must be restricted to lie in an elementary abelian $2$-group in order to obtain a balanced splittable Hadamard matrix according to \cref{Thm:PDS}. 

\item
The construction of \cref{Thm:PDS} employs a similar technique to that used in \cite{GR} to construct a set of mutually unbiased bases from a semiregular relative difference set.

\item
All parameter constraints on a $\BSHM(2^r,\ell,a,b)$ (in particular, 
\cref{Prop:SRG2} and \cref{Cor:ell2}) apply via \cref{Thm:PDS} to partial difference sets of cardinality $\ell$ in $\Z_2^r$ whose nonprincipal character sums take the values $a$ and~$b$.
\end{enumerate}
\end{remark}

\begin{corollary}\label{Cor:PDSAandB}
Let $1 < \ell < 2^r-2$ and $b \not \in \{-a,-a-2\}$.
Suppose $D$ is a partial difference set of cardinality~$\ell$ in $\Z_2^r$ whose nonprincipal character sums take the values $a$ and~$b$, and that $1_{\Z_2^r} \not \in D$. Then there exists a Type~1 $\BSHM(2^r,\ell,a,b)$ and there exists a Type~2 $\BSHM(2^r,\ell+1,a+1,b+1)$.
\end{corollary}

\begin{proof}
By \cref{Thm:PDS}, $C(\Z_2^r)$ is a Type 1 $\BSHM(2^r,\ell,a,b)$ with respect to $C(D)$ that contains the all-ones row. Then by \cref{Prop:add-subtract}~$(ii)$, there exists a Type 2 $\BSHM(2^r,\ell+1,a+1,b+1)$.
\end{proof}

We now apply \cref{Cor:PDSAandB} to families of regular partial difference sets in elementary abelian $2$-groups (obtained in each case from the indicated reference) in order to construct
new infinite families of balanced splittable Hadamard matrices.
This result was used to determine existence for 14 of the 29 entries in each of \cref{Tab:survivebne-aType1,Tab:survivebne-aType2},
leaving no open cases for the existence of a primitive $\BSHM(n,\ell,a,b)$ with $n \in \{64,256\}$.

\begin{corollary}\label{Cor:PDS}
There exists a Type~1 $\BSHM(n,\ell,a,b)$ and a Type~2 $\BSHM(n,\ell+1,a+1,b+1)$ for each of the following parameter sets $(n,\ell,a,b)$, provided $1 < \ell < n-2$ and $b \not \in \{-a,-a-2\}$:
\begin{enumerate}[$(i)$]
\item
$(2^{2m},s(2^m-1),2^m-s,-s)$, 
where $m \ge 1$ and $1 \le s \le 2^m+1$	{\rm\cite[Example~2.3.3]{Ma}}

\item
$\big(2^{3m},\, (2^{m+s}-2^m+2^s)(2^m-1),\, 2^m-2^s,\, 2^m-2^s-2^{m+s} \big)$, 
where $1 \le s < m$ {\rm\cite[Example~9.4]{Ma}}

\item
$\big(2^{2sm},\, 2^{(s-1)m}(2^{m-1}-1)(2^{sm}-1),\, 2^{(s-1)m}(2^{m-1}+1),\, -2^{(s-1)m}(2^{m-1}-1) \big)$ 
and \\
$\big(2^{2sm},\, 2^{(s-1)m}(2^{m-1}-1)(2^{sm}+1),\, 2^{(s-1)m}(2^{m-1}-1),\, -2^{(s-1)m}(2^{m-1}+1) \big)$, 
where $s,m \ge 1$ {\rm\cite[Example~9.6]{Ma}}

\item
$\big(2^{12m},\, t(2^{2m}-1)(2^{2m}+2^m+1)(2^{6m}+1),\, t(2^{2m}-1)(2^{2m}+2^m+1),\, t(2^{2m}-1)(2^{2m}+2^m+1)-2^{6m} \big)$, 
where $m \ge 1$ and $1 \le t \le 2^m(2^m-1)$ {\rm\cite[Theorem~9.7]{Ma}}

\item
$\big(2^{(4s+2)m},\, t\frac{(2^{(2s+1)m}-1)(2^{(2s-1)m}+1)}{2^m+1},\, 2^{(2s+1)m} - t\frac{2^{(2s-1)m}+1}{2^m+1},\, -t\frac{2^{(2s-1)m}+1}{2^m+1} \big)$ and \\
$\big(2^{4sm},\, t\frac{2^{4sm}-1}{2^m+1},\, t\frac{2^{2sm}-1}{2^m+1},\, \frac{2^{2sm}(t-1-2^m)-t}{2^m+1} \big)$,
where $s, m \ge 1$ and $1 \le t \le 2^m+1$ {\rm\cite[Example~10.5]{Ma}}

\item
$\big(2^{(4s+2)m},\, \frac{2^m(2^{2sm}-1)}{2^m+1}(2^{(2s+1)m}+1),\, \frac{2^m(2^{2sm}-1)}{2^m+1},\, -\frac{2^m(2^{(2s+1)m}+1)}{2^m+1} \big)$, 
where $s,m \ge 1$ {\rm\cite[p. 282]{Mo}}

\item
$\big(2^{4sm},\, \frac{(2^{2sm}-1)^2}{t},\, \frac{(t-1)2^{2sm}+1}{t},\, -\frac{2^{2sm}-1}{t} \big)$ \,\,
and \,\,
$\big(2^{4sm},\, \frac{2^{4sm}-1}{t},\, \frac{2^{2sm}-1}{t},\, -\frac{(t-1)2^{2sm}+1}{t} \big)$, 
where $m \ge 2$ and $t \ge 3$ is odd and $s$ (if it exists) is the smallest positive integer for which $t \mid (2^s+1)$ {\rm\cite[Corollary 3.7]{MX}} 

\item
$\big(2^{2m},\, (2^{m-s}-1)(2^{m}-1),\, -2^{m-s}+1+2^m,\, 1-2^{m-s} \big)$ 
and \\
$\big(2^{2m},\, (2^{m-s}-1)(2^{m}+1),\, 2^{m-s}-1,\, 2^{m-s}-1-2^m \big)$, 
where $m \ge 3$ and $s \mid m$ {\rm\cite[Corollary 3.8]{MX}}. 
\end{enumerate}
\end{corollary}

Call a representation of a matrix $H$ in the form $\begin{pmatrix} H_1 \\ H_2 \\ \vdots \\ H_w \end{pmatrix}$
a \emph{row decomposition of $H$ into submatrices $H_1, H_2, \dots H_w$}.
\cref{Res:twin} demonstrates the existence of a family of Hadamard matrices $H$ admitting a row decomposition into submatrices $H_1, H_2, H_3$ so that $H$ has the balanced splittable property with respect to each~$H_i$.
We shall significantly extend this result by making a connection to special packings of partial difference sets: groups whose nonzero elements can be partitioned into subsets so that every subset, and every union of subsets, is a partial difference set.

\begin{definition}\label{Defn:PDSpacking}
Let $G$ be a finite abelian group. A \emph{$(\delta,t)$ partial difference set packing of $G$ with respect to $(a_1,\dots,a_t)$} is a
partition of $G \setminus \{1_G\}$ into subsets $D_1, \dots, D_t$ 
such that each $D_i$ is a partial difference set in $G$ with nonprincipal character sums $a_i$ or $\delta+a_i$, 
and for each subset $I$ of $\{1,\dots,t\}$ the set $\bigcup_{i \in I}D_i$ is a partial difference set in $G$ 
with nonprincipal character sums $\sum_{i \in I}a_i$ or $\delta+\sum_{i \in I}a_i$.
\end{definition}

\begin{theorem}\label{Thm:PDSpackingtoBSHM}
Suppose $D_1, \dots, D_t$ is a $(\delta,t)$ partial difference set packing in~$\Z_2^r$ with respect to $(a_1,\dots,a_t)$.
Partition $\{1,\dots,t\}$ into subsets $I_1 \dots, I_w$,
and for each $u$ in $\{1,\dots,w\}$ let
$\ell_u = \sum_{i \in I_u} |D_i|$ and 
$\alpha_u = \sum_{i \in I_u} a_i$ and 
$H_u = C\Big(\bigcup_{i \in I_u} D_i\Big)$. 
Let $j \in \{1,\dots,w\}$.
Then the order $2^r$ matrix $C(\Z_2^r)$ admits a row decomposition into 
submatrices $\begin{pmatrix} \vo^T \end{pmatrix}, H_1, \dots, H_w$ such that
the matrix $C(\Z_2^r)$ is simultaneously a
$\BSHM( 2^r, \ell_j+1, \alpha_j +1, \delta +\alpha_j +1)$ 
with respect to~$\begin{pmatrix} \vo^T \\ H_j \end{pmatrix}$ and a
$\BSHM( 2^r, \ell_u, \alpha_u, \delta +\alpha_u)$ 
with respect to $H_u$ for each $u \ne j$.
\end{theorem}

\begin{proof}
By \cref{Defn:PDSpacking}, we have $1_{\Z_2^r} \not \in D_u$ for each~$u$.
Therefore by \cref{Defn:C(D)}, the matrix $C(\Z_2^r)$ admits a row decomposition into submatrices 
$\begin{pmatrix} \vo^T \end{pmatrix}, H_1, \dots, H_w$.
By \cref{Defn:PDSpacking,Thm:PDS}, the matrix $C(\Z_2^r)$ is a
$\BSHM(2^r, \ell_u, \alpha_u, \delta+\alpha_u)$ with respect to $H_u$ for each $u$ in $\{1,\dots,w\}$.
Apply \cref{Rem:addrow}.
\end{proof}

\begin{remark}
By \cref{Thm:PDS}, the balanced splittable Hadamard matrices constructed in \cref{Thm:PDSpackingtoBSHM}
for $u \ne j$ have Type 1 provided $\alpha_u \ne -(\delta+\alpha_u)$,
and 
that for $u=j$ has Type 2 provided $\alpha_j +1 \ne -(\delta+\alpha_j+1)$.
\end{remark}

The LP-packings and NLP-packings of partial difference sets introduced in \cite{JL21} provide examples of $(\delta,t)$ partial difference set packings.

\begin{definition}{\cite[Definition 3.1 and Lemma 2.5]{JL21}}
Let $t > 1$ and $c > 0$ be integers.
Let $G$ be an abelian group of order $t^2c^2$, and let $U$ be a subgroup of $G$ of order~$tc$.
A \emph{$(c,t)$ LP-packing in $G$ relative to $U$} is a partition of $G \setminus U$ into 
$t$ partial difference sets in $G$, each of which have cardinality $c(tc-1)$ and nonprincipal character sums $-c$ or $(t-1)c$.
\end{definition}

\begin{lemma}{\rm{\cite[Lemma 3.9]{JL21}}}\label{Lemma:LPtoPDSpacking}
Suppose that $D_1,\dots, D_t$ is a $(c,t)$ LP-packing in an abelian group $G$ of order $t^2c^2$ 
relative to a subgroup $U$ of order $tc$.
Then $U \setminus \{1_G\}, D_1, \dots, D_t$ is a $(tc, t+1)$ partial difference set packing in~$G$ with respect to $(-1,-c,-c,\dots,-c)$.
\end{lemma}

\begin{definition}{\cite[Definition 6.1 and Lemma 2.5]{JL21}}
Let $t > 1$ and $c > 0$ be integers.
Let $G$ be an abelian group of order $t^2c^2$.
A \emph{$(c,t-1)$ NLP-packing in $G$} is a partition of $G \setminus \{1_G\}$ into $t$ partial difference sets in $G$, of which
$t-1$ have cardinality $c(tc+1)$ and nonprincipal character sums $c$ or $-(t-1)c$,
and one has cardinality $(c-1)(tc+1)$ and nonprincipal character sums $c-1$ or $-tc+c-1$.
\end{definition}

\begin{lemma}{\rm{\cite[Lemma 6.4 and Remark 6.5(i)]{JL21}}}\label{Lemma:NLPtoPDSpacking}
Suppose that $D_0,D_1,\dots, D_{t-1}$ is a $(c,t-1)$ NLP-packing in an abelian group $G$ of order $t^2c^2$,
where $D_0$ is the exceptional subset.
Then $D_0, D_1, \dots, D_{t-1}$ is a $(-tc, t)$ partial difference set packing in~$G$ with respect to $(c-1,c,c,\dots,c)$.
\end{lemma}

We now apply \cref{Thm:PDSpackingtoBSHM} to $(\delta,t)$ partial difference set packings drawn from the literature \cite{JL21,P,PDS} in order to produce infinite families of Hadamard matrices admitting a row decomposition so that the balanced splittable property holds simultaneously with respect to every union of the submatrices of the decomposition.
We believe this approach to the construction of balanced splittable Hadamard matrices (or equivalent objects) to be entirely new.

\begin{corollary}\label{Cor:PDSpacking}
Let $w,t$ be positive integers, 
and let $m$ and $a_1, \dots, a_t$ be integers. 
Let $I_1, \dots, I_w$ be a partition of $\{1,\dots,t\}$, and let $j \in \{1,\dots,w\}$.
Then, for the parameter values in each of the cases $(i)$ to $(v)$ displayed below, 
the order $2^r$ matrix $C(\Z_2^r)$ admits a row decomposition into 
submatrices $\begin{pmatrix} \vo^T \end{pmatrix}, H_1, \dots, H_w$ such that
$C(\Z_2^r)$ is simultaneously a
$\BSHM(2^r, \ell_j+1, \alpha_j +1, \delta +\alpha_j +1)$ 
with respect to~$\begin{pmatrix} \vo^T \\ H_j \end{pmatrix}$ and a
$\BSHM(2^r, \ell_u, \alpha_u, \delta +\alpha_u)$ 
with respect to $H_u$ for each $u \ne j$.
\begin{table}[ht!]
\scriptsize
\begin{center}
$
\begin{array}{|rl|c|c|c|c|c|c|} 		
\hline
        & 				& r	& t	& (a_1,\dots,a_t)		& \alpha_u 		& \ell_u 		& \delta 	\\ \hline
(i)	& m \ge 1			& 2m	& 2^m+1	& (-1,\dots,-1)			& -|I_u|		& (2^m-1)|I_u|		& 2^m 		\\ 
(ii)	& m = 0 \mbox{ or } m \ge 2 	& 2m+4	& 4	& (2^m-1,2^m,2^m,2^m)		& \sum_{i \in I_u} a_i	& (2^{m+2}+1)\alpha_u	& -2^{m+2} 	\\
(iii)	& m \ge 3		 	& 2m+4	& 4	& (2^m+2,2^m-1,2^m-1,2^m-1)	& \sum_{i \in I_u} a_i	& (2^{m+2}+1)\alpha_u	& -2^{m+2} 	\\
(iv)	& m \ge 3		 	& 2m+4	& 4	& (2^m,2^m+1,2^m-1,2^m-1)	& \sum_{i \in I_u} a_i	& (2^{m+2}+1)\alpha_u	& -2^{m+2} 	\\ 
(v)	& 				& 6	& 3	& (2,2,3)			& \sum_{i \in I_u} a_i	& 9 \alpha_u		& -8		\\ [0.5ex] \hline
\end{array}
$
\end{center}
\normalsize
\end{table}
\end{corollary}

\begin{proof}
For each of the displayed sets of parameter values, we provide a $(\delta,t)$ partial difference set packing in $\Z_2^r$ with respect to $(a_1,\dots,a_t)$, 
calculate $\alpha_u = \sum_{i \in I_u} a_i$ and $\ell_u = \sum_{i \in I_u} |D_i|$, 
and apply \cref{Thm:PDSpackingtoBSHM}.

\begin{enumerate}[$(i)$]
\item 
By \cite[Theorem~5.3]{JL21}, 
there is a $(1,2^m)$ 
LP-packing in $\Z_2^{2m}$ relative to $\Z_2^m$. 
Therefore by \cref{Lemma:LPtoPDSpacking}, there is a $(2^m,2^m+1)$ 
partial difference set packing $D_1,\dots,D_{2^m+1}$ in $\Z_2^{2m}$ with respect to $(a_1,\dots,a_{2^m+1}) = (-1,-1,\dots,-1)$, where 
$|D_i| = 2^m-1$ 
for each~$i$.
Then $\alpha_u = \sum_{i \in I_u} a_i = -|I_u|$ for each~$u$ and
$\ell_u = \sum_{i \in I_u} |D_i| = (2^m-1) |I_u|$.

\item 
By \cite[Corollary 6.13~$(i)$]{JL21}, 
there is an $(2^m,3)$ 
NLP-packing in $\Z_2^{2m+4}$.
Therefore by \cref{Lemma:NLPtoPDSpacking}, there is a $(-2^{m+2},4)$ 
partial difference set packing $D_1,D_2,D_3,D_4$ in $\Z_2^{2m+4}$ with respect to $(a_1,a_2,a_3,a_4) = (2^m-1,2^m,2^m,2^m)$, where 
$|D_i| = (2^{m+2}+1)a_i$ for each~$i$.
Then $\ell_u = \sum_{i \in I_u} |D_i| = (2^{m+2}+1) \alpha_u$ for each~$u$.

\item 
By \cite[Corollary 6.1 and page 276 lines 1--3]{P} 
and \cref{Prop:PDSchar},
there is a $(-2^{m+2},4)$ partial difference set packing $D_1,D_2,D_3,D_4$ in $\Z_2^{2m+4}$ 
with respect to $(a_1,a_2,a_3,a_4)$ $= (2^m+2,2^m-1,2^m-1,2^m-1)$, where
$|D_i| = (2^{m+2}+1)a_i$ for each~$i$.
Then $\ell_u 
= (2^{m+2}+1) \alpha_u$ for each~$u$.

\item 
By \cite[Corollary 6.2 and page 277 lines 1--3]{P} (and a clarification provided in \cite{Ppc}) 
and \cref{Prop:PDSchar},
there is a $(-2^{m+2},4)$ partial difference set packing $D_1,D_2,D_3,D_4$ in $\Z_2^{2m+4}$ 
with respect to $(a_1,a_2,a_3,a_4) = (2^m,2^m+1,2^m-1,2^m-1)$, where
$|D_i| = (2^{m+2}+1) a_i$ for each~$i$.
Then $\ell_u 
= (2^{m+2}+1) \alpha_u$ for each~$u$.

\item 
By \cite[Example 3.1]{PDS} and \cref{Prop:PDSchar},  
there is a $(-8,3)$ partial difference set packing $D_1,D_2,D_3$ in $\Z_2^6$ 
with respect to $(a_1,a_2,a_3) = (2,2,3)$, where
$|D_i| = 9 a_i$ for each~$i$.
Then $\ell_u = 9 \alpha_u$ for each~$u$.
\qedhere
\end{enumerate}

\end{proof}

\begin{remark}
\mbox{}
\begin{enumerate}[$(i)$]
\item
The construction of \cref{Res:twin}, involving a Sylvester-type Hadamard matrix of order $2^{2m}$,
occurs as the special case $w = 3$, $j=1$, $|I_1| = 1$, $|I_2| = |I_3| = 2^{m-1}$ of \cref{Cor:PDSpacking}~$(i)$.

\item
The result of \cref{Cor:PDS}~$(i)$ (which is obtained from \cite[Example 2.3.3]{Ma}) is implied by the result of
\cref{Cor:PDSpacking}~$(i)$ with respect to a single submatrix~$H_u$, by taking $|I_u| = s$ for $u \ne j$. 
\end{enumerate}

\end{remark}

\end{subsection}

\end{section}

\begin{section}{Constructions and restrictions for the imprimitive case}\label{Sec:imprimitive}

In this section, we develop the study of imprimitive balanced splittable Hadamard matrices.
\cref{Prop:imprimitiveb=-a,Prop:imprimitivebne-a} show that, up to application of the switching transformation \eqref{Eq:switching} and interchange of $a,b$ and (for $b=-a$) negation of columns, we need consider only two cases for an imprimitive balanced splittable Hadamard matrix~$H$:
\begin{enumerate}[$(i)$]
    \item $H$ is a Type~1 $\BSHM(4rs,4s-1,4s-1,-1)$ for some integers $r \ge 2$, $s \ge 1$
    \item $H$ is a Type~2 $\BSHM(8rs,4s,4s,0)$ for some integers $r, s \ge 1$.
\end{enumerate}
We examine case $(ii)$ in \cref{Subsec:8rs}, using a Kronecker product construction to obtain new infinite families.
We examine case $(i)$ in \cref{Subsec:4rs}, using a connection to case $(ii)$ to construct new infinite families and then using structural constraints to further restrict the possible parameter values when $r$ is odd.

\begin{subsection}{Type 2 BSHM$(8rs,4s,4s,0)$}\label{Subsec:8rs}

Recall the notation $A \otimes B$ for the Kronecker product of matrices $A$ and $B$ (see \cref{Subsec:KS}).

\begin{proposition}\label{Prop:Structure}
Suppose $H = \begin{pmatrix}
    H_1 \\
    H_2
    \end{pmatrix}$ is a $\BSHM(8rs,4s,4s,0)$ with respect to~$H_1$.
Then there exists a Hadamard matrix $L$ of order $4s$, and the columns of $H$ can be reordered so that $H_1 = \vo^T \otimes L$.
\end{proposition}

\begin{proof} 
By \cref{Prop:SRG2}~(B4) with $(a,b)=(4s,0)$, we have $k_{4s} = 2r-1$. 
Therefore by \cref{Defn:kai} each column of $H_1$ has length $4s$ and is identical to (has dot product $4s$ with) exactly $2r-1$ other columns of~$H_1$, and is orthogonal to (has dot product $0$ with) all other columns of~$H_1$. Therefore $H_1$ has exactly $\frac{8rs}{2r} = 4s$ distinct column types $\vc_1, \dots, \vc_{4s}$, each occurring with multiplicity~$2r$. The columns $\vc_1, \dots, \vc_{4s}$ are mutually orthogonal and so form a Hadamard matrix $L$ of order $4s$, and the columns of $H$ can be reordered so that 
$H_1 = \vo^T \otimes L 
     = \underbrace{\begin{pmatrix} L & L & \dots & L \end{pmatrix}}_{2r}$.
\end{proof}

The following construction modifies the Kronecker product construction used to establish \cref{Res:Constructions}~$(iii)$. 

\begin{proposition}\label{Prop:Product}
Suppose there exists a $\BSHM(n,\ell,\ell,0)$ and a Hadamard matrix of order~$m$. Then there exists a $\BSHM(nm,\ell m,\ell m,0)$.
\end{proposition}

\begin{proof}
Let $H$ be a $\BSHM(n,\ell,\ell,0)$ with respect to a submatrix~$H_1$, and let
$K$ be an order $m$ Hadamard matrix. 
We shall show that $H \otimes K$ is a $\BSHM(nm,\ell m,\ell m,0)$ with respect to
$H_1 \otimes K$.

The matrix $H_1 \otimes K$ is an
$\ell m \times nm$ submatrix of the order $nm$ Hadamard matrix $H \otimes K$.
By \cref{Rem:A}, we have
$H_1^TH_1 = \ell (I_n + A)$, where $A$ is a symmetric matrix over $\{0,1\}$ with zero diagonal.
Therefore
\begin{align*}
(H_1 \otimes K)^T
(H_1 \otimes K)
  &=\ell (I_{n}+A) \otimes m I_m,
\end{align*}
and all off-diagonal entries of the symmetric matrix $(I_{n}+A) \otimes I_m$ lie in $\{0,1\}$.
\end{proof}

\begin{theorem}\label{Thm:b=0}
There exists a $\BSHM(8rs,4s,4s,0)$ in each of the following cases:
\begin{enumerate}[$(i)$]
\item
there exist Hadamard matrices of order $2r$ and $4s$

\item
there exist Hadamard matrices of order~$4r$ and~$2s$.
\end{enumerate}
\end{theorem}

\begin{proof}
\mbox{}
\begin{enumerate}[$(i)$]
\item
Apply \cref{Res:Constructions}~$(iii)$ to Hadamard matrices of order $2r$ and~$4s$.

\item
Apply \cref{Res:Constructions}~$(v)$ to a Hadamard matrix of order $4r$ to obtain a $\BSHM(4r,2,2,0)$.
Combine this with a Hadamard matrix of order~$2s$ using \cref{Prop:Product}.
\qedhere
\end{enumerate}
\end{proof}

\begin{remark}
All parameter sets for which a $\BSHM(8rs,4s,4s,0)$ is known to exist are constructed in \cref{Thm:b=0}. Those in $(i)$ were previously known from \cref{Res:Constructions}~$(iii)$, whereas those in $(ii)$ are new.
\end{remark}

\cref{Tab:ellzero} shows the parameter sets for which the existence of a $\BSHM(8rs,4s,4s,0)$ is not determined by \cref{Thm:b=0}, for $r , s \le 8$.

\begin{table}[ht!]
\caption{Open cases for a $\BSHM(8rs,4s,4s,0)$ with $r,s \le 8$.}
\begin{center}
\begin{tabular}{|c|c|c|} 		
\hline
$r$	& $s$	& $(8rs,4s,4s,0)$ \\ \hline
3	& 3	& $( 72,12,12,0)$ \\ 
5	& 3	& $(120,12,12,0)$ \\ 
3	& 5	& $(120,20,20,0)$ \\ 
7	& 3	& $(168,12,12,0)$ \\ 
3	& 7	& $(168,28,28,0)$ \\ 
5	& 5	& $(200,20,20,0)$ \\ 
7	& 5	& $(280,20,20,0)$ \\ 
5	& 7	& $(280,28,28,0)$ \\ 
7	& 7	& $(392,28,28,0)$ \\ \hline
\end{tabular}
\end{center}
\label{Tab:ellzero}
\end{table}

Assuming the Hadamard matrix conjecture holds (see \cref{Sec:intro}), 
\cref{Thm:b=0} gives the following result.

\begin{corollary} 
Assume the Hadamard matrix conjecture holds. Then there exists a $\BSHM(8rs,4s,4s,0)$ for all positive integers $r,s$ except possibly when $r,s$ are both odd and greater than~$1$.
\end{corollary}

\end{subsection}

\begin{subsection}{Type 1 BSHM$(4rs,4s-1,4s-1,-1)$}\label{Subsec:4rs}

We begin with a structural result similar to that of \cref{Prop:Structure}.

\begin{proposition}\label{Prop:Structure2}
    Suppose $H = \begin{pmatrix} H_1 \\ H_2 \end{pmatrix}$ 
is a $\BSHM(4rs,4s-1,4s-1,-1)$ with respect to~$H_1$.  
Then there exists a Hadamard matrix $\begin{pmatrix}\vo^T \\ L \end{pmatrix}$ of order~$4s$, and the columns of $H$ can be reordered so that $H_1 = \vo^T \otimes L$.
\end{proposition}

\begin{proof}
We are given that $(a,b) = (4s-1,-1)$, and so the dot product of every two distinct columns of the $4s \times 4rs$ matrix $H'_1 = \begin{pmatrix}\vo^T \\ H_1 \end{pmatrix}$ lies in $\{4s,0\}$.
By \cref{Prop:SRG2}~(A4) we have $k_{4s-1} = r-1$, and therefore each column of $H'_1$ is identical to (has dot product $4s$ with) exactly $r-1$ other columns of~$H'_1$, and is orthogonal to (has dot product $0$ with) all other columns of~$H'_1$. 
Therefore $H'_1$ has exactly $\frac{4rs}{r} = 4s$ distinct column types $\vc_1, \dots, \vc_{4s}$, each occurring with multiplicity~$r$. The columns $\vc_1, \dots, \vc_{4s}$ are mutually orthogonal and so form a Hadamard matrix 
$\begin{pmatrix}\vo^T \\ L \end{pmatrix}$ of order~$4s$, and the columns of 
$\begin{pmatrix} H'_1 \\ H_2 \end{pmatrix}$ can be reordered so that 
$H'_1 = \vo^T \otimes \begin{pmatrix} \vo^T \\ L \end{pmatrix}
= \begin{pmatrix} \vo^T \\ \vo^T \otimes L \end{pmatrix}$.
Therefore the columns of $H$ can be reordered so that 
$H_1 = \vo^T \otimes L
     = \underbrace{ \begin{pmatrix} L & L & \dots & L \end{pmatrix}}_r$.
\end{proof}

We now consider the construction of a $\BSHM(4rs,4s-1,4s-1,-1)$. We shall see that the cases $r$ even and $r$ odd behave differently.
We begin with a relation between a $\BSHM(4rs,4s,4s,0)$ and a $\BSHM(4rs,4s-1,4s-1,-1)$, making use of \cref{Rem:addrow}.

\begin{proposition}\label{Prop:0to-1}
For positive integers $r, s$, there exists 
a $\BSHM(4rs,4s,4s,0)$ if and only if there exists 
a $\BSHM(4rs,4s-1,4s-1,-1)$ containing the all-ones row.
\end{proposition}

\begin{proof}
It is sufficient to prove the implication in the forward direction, because the reverse direction follows directly from \cref{Prop:add-subtract}~$(ii)$. 
Suppose that $H = \begin{pmatrix} H_1 \\ H_2 \end{pmatrix}$ 
is a $\BSHM(4rs,4s,4s,0)$ with respect to~$H_1$.
By \cref{Prop:Structure}, we may reorder the columns of $H$ so that 
$H_1 = \vo^T \otimes L$ for some order $4s$ Hadamard matrix $L$ having columns $\vc_1, \dots, \vc_{4s}$. 
Transform~$L$, by negating its columns as necessary, to an order $4s$ Hadamard matrix $\begin{pmatrix} L' \\ \vo^T \end{pmatrix}$.
Then the matrix $K_1 = \vo^T \otimes \begin{pmatrix} L' \\ \vo^T \end{pmatrix}$
is obtained from $H_1$ by either negating or leaving unchanged all $r$ occurrences of $\vc_i$ (for each $i$ independently), and so the dot product of two distinct columns of $K_1$ is identical to the dot product of the same two columns of $H_1$ (namely $4s$ or $0$).
Therefore, writing 
$K_1 = \begin{pmatrix} H'_1 \\ \vo^T \end{pmatrix}$, 
 the matrix
$\begin{pmatrix} K_1 \\ H_2 \end{pmatrix}
=\begin{pmatrix} H'_1 \\ \vo^T \\ H_2 \end{pmatrix}$ 
is a $\BSHM(4rs,4s,4s,0)$ with respect to $K_1$.
This matrix contains the all-ones row, and by \cref{Rem:addrow} is a $\BSHM(4rs,4s-1,4s-1,-1)$ with respect to~$H'_1$.
\end{proof}

\cref{Prop:0to-1} allows us to transform each imprimitive balanced splittable Hadamard matrix with $b=0$ constructed in \cref{Thm:b=0} to an imprimitive balanced splittable Hadamard matrix with $b=-1$.

\begin{corollary} \label{Cor:b=-1}
\mbox{}
\begin{enumerate}[$(i)$]
\item 
There exists a $\BSHM(4rs,4s-1,4s-1,-1)$ in each of the following cases:
\begin{enumerate}[$(a)$]
\item 
there exist Hadamard matrices of order $r$ and~$4s$

\item
there exist Hadamard matrices of order~$2r$ and~$2s$.
\end{enumerate}

\item
Let $r$ be odd. Then there is no $\BSHM(4rs,4s-1,4s-1,-1)$ containing the all-ones row.

\end{enumerate}
\end{corollary}

\begin{proof}
\mbox{}
\begin{enumerate}[$(i)$]
\item Apply \cref{Prop:0to-1} to \cref{Thm:b=0}.

\item By \cref{Prop:imprimitivebne-a}, a $\BSHM(4rs,4s,4s,0)$ does not exist for $r$ odd. Apply \cref{Prop:0to-1}.
\qedhere

\end{enumerate}
\end{proof}

\begin{remark} \label{Rem:skew-type}
\cref{Cor:b=-1}~$(ii)$ does not hold if the condition on the presence of the all-ones row is removed: by \cref{Res:Constructions}~$(vi)$, 
a $\BSHM(4s(4s-1),4s-1,4s-1,-1)$ exists provided $4s$ is the order of a skew-type Hadamard matrix. 
A conjecture attributed to Seberry \cite[p.~274]{Handbook} states that a skew-type Hadamard matrix exists for each order $4s$, and the conjecture is known to hold for all $s < 47$ \cite[p.~275]{Handbook}.
\end{remark}

We now derive some further parameter restrictions when $r$ is odd.

\begin{proposition}\label{Prop:ratiobound}
Suppose $H = \begin{pmatrix} H_1 \\ H_2 \end{pmatrix}$ is a nontrivial 
$\BSHM(4rs,4s-1,4s-1,-1)$ with respect to $H_1$, where $r$ is odd. 
Then $r^2$ is the sum of $4rs-4s+1$ odd squares, and $r \ge 4s-1$.
\end{proposition}

\begin{proof}
By \cref{Prop:Structure2}, the columns of $H$ can be reordered so that
\begin{align*}
    H_1 =
    \underbrace{
    \begin{pmatrix}
    L & L & \dots & L
    \end{pmatrix}}_r
\end{align*}
for some $(4s-1) \times 4s$ matrix~$L$.
For $1 \le i \le r$, let $\begin{pmatrix}\vu_i \\ \vc_i \end{pmatrix}$ be the $(4is+1)^\text{th}$ column of~$H$, where $\vu_i$ is contained in $H_1$ and $\vc_i$ is contained in~$H_2$.
The repeating structure of $H_1$ gives 
\[
\vu_i \cdot \vu_j = 4s-1 \quad \mbox{for all $i,j$}.
\]
Since the columns of the Hadamard matrix $H$ are orthogonal, this gives the following relationship between the columns of $H_2$:
\[
\vc_i \cdot \vc_j = \begin{cases}
  4rs-4s+1 & \mbox{for $i=j$}, \\
  -(4s-1)  & \mbox{for $i \ne j$}.
\end{cases}
\]

Now let $\vv = \sum_{i=1}^r \vc_i$. Then
\begin{align*}
\vv \cdot \vv
 &= \sum_{i = 1}^r \vc_i \cdot \vc_i + 
\sum_{i \ne j} \vc_i \cdot \vc_j \\
 &= r (4rs-4s+1) - r(r-1)(4s-1) \\
 &= r^2.
\end{align*}
Each of the $4rs-4s+1$ entries of $\vv$ is the sum of $r$ terms $\pm 1$, and so is odd because $r$ is odd by assumption. 
Therefore $\vv \cdot \vv = r^2$ is the sum of $4rs-4s+1$ odd squares. 
This implies that $r^2 \ge 4rs-4s+1$,
which rearranges as
\[
(r-1)(r+1-4s) \ge 0.
\]
Since $H$ is nontrivial, we have $r \ge 1$ and therefore $r \ge 4s-1$.
\end{proof}

\begin{remark}
The bound on $r$ in \cref{Prop:ratiobound} is tight, by the construction of \cref{Res:Constructions}~$(vi)$.
\end{remark}

\cref{Tab:ellminus} shows the parameter sets for which the existence of a $\BSHM(4rs,4s-1,4s-1,-1)$ is not determined by \cref{Res:Constructions}~$(vi)$ and \cref{Cor:b=-1} and \cref{Prop:ratiobound}, for $2 \le r \le 12$ and $s \le 8$.

\begin{table}[ht!]
\caption{Open cases for a $\BSHM(4rs,4s-1,4s-1,-1)$ with $2 \le r \le 12$ and $s \le 8$.}
\begin{center}
\begin{tabular}{|c|c|c|} 		
\hline
$r$	& $s$	& $(4rs,4s-1,4s-1,-1)$ \\ \hline
5       & 1	& $( 20, 3, 3,-1)$ \\
7	& 1	& $( 28, 3, 3,-1)$ \\
9   	& 1	& $( 36, 3, 3,-1)$ \\
11	& 1	& $( 44, 3, 3,-1)$ \\
9	& 2	& $( 72, 7, 7,-1)$ \\
6	& 3 	& $( 72,11,11,-1)$ \\
11	& 2	& $( 88, 7, 7,-1)$ \\
10	& 3	& $(120,11,11,-1)$ \\
6	& 5 	& $(120,19,19,-1)$ \\
6	& 7 	& $(168,27,27,-1)$ \\
10	& 5	& $(200,19,19,-1)$ \\
10	& 7	& $(280,27,27,-1)$ \\ \hline
\end{tabular}
\end{center}
\label{Tab:ellminus}
\end{table}

\cref{Res:Constructions}~$(vi)$ and \cref{Cor:b=-1} and \cref{Prop:ratiobound} combine to give the following result. (The existence of a skew-type Hadamard matrix is discussed in \cref{Rem:skew-type}.)

\begin{corollary} 
Assume the Hadamard matrix conjecture holds. Then the existence of a $\BSHM(4rs,4s-1,4s-1,-1)$ for integers $r \ge 2$ and $s \ge 1$ is determined by:

\begin{center}
\footnotesize
\begin{tabular}{|l|l|} 		
\hline
$r=2$ or $r \equiv 0 \pmod{4}$		& exists					\\ \hline
$r \equiv 2 \pmod{4}$ and $r > 2$	& exists if $s=1$ or $s$ even, 			\\
					& otherwise open				\\ \hline
$r$ odd 				& does not exist if $r < 4s-1$,				\\
					& exists if $r=4s-1$ and $4s$ is order of skew-type Hadamard matrix, \\ 
					& otherwise open 				\\ \hline
\end{tabular}
\normalsize
\end{center}
\end{corollary}

\end{subsection}

\end{section}

\begin{section}{Open questions}\label{Sec:future}

We propose some open questions for future research. 

\begin{enumerate}[$(i)$]
\item 
\cref{Prop:SRG1,Prop:SRG2} show that a balanced splittable Hadamard matrix is associated with a strongly regular graph, and we have restricted the possible forms of these graphs as summarized in \cref{Tab:summary}.
We are grateful to a reviewer for proposing the converse question:
under what conditions can a given strongly regular graph be associated with a balanced splittable Hadamard matrix? 

\item
Does there exist a $\BSHM(n,\ell,a,-a)$ where $(n,\ell,a)$ do not take the form $(4u^2,2u^2-u,u)$ for some integer~$u$
(as in \cref{Res:Equi})? From \cref{Tab:surviveb=-a}, the smallest open case occurs at $(n,\ell,a) = (288,42,6)$.
As noted after \cref{Def:ETF}, a $\BSHM(n,\ell,a,-a)$ is equivalent to an $\ell \times n$ real flat ETF that is a submatrix of an order $n$ Hadamard matrix. 
It is therefore interesting to note the existence of a real flat ETF with size not of the form $(2u^2-u) \times 4u^2$, as discussed in \cite[p. 297]{FJMP}.

\item
A $\BSHM(4u^2,2u^2-u,u,-u)$ exists when $u$ is the order of a Hadamard matrix (\cref{Res:Equi}), but does not exist when $u$ is odd (\cref{Res:b=-a}~$(i)$).
Does there exist a $\BSHM(4u^2,2u^2-u,u,-u)$ for $u \equiv 2 \pmod{4}$ and $u > 2$? From \cref{Tab:surviveb=-a}, the smallest open case occurs at $u=6$, for a $\BSHM(144,66,6,-6)$ (see also \cite[p.~2050]{KPS}). We note that a real flat $66 \times 144$ ETF is known to exist \cite[Corollary 1]{FJMP}.
 
\item
By Results~\ref{Res:ETF} and \ref{Res:NonEqui}~$(ii)$, the parameter $n$ for a $\BSHM(n,\ell,a,b)$ is bounded from above by a quadratic function in~$\ell$. These results were derived for matrices with real entries, not necessarily restricted to lie in~$\{1,-1\}$. Can the growth rate of the upper bound be improved by including the condition that the entries of a balanced splittable Hadamard matrix must lie in $\{1,-1\}$?

\item
\cref{Cor:PDSpacking} uses $(\delta,t)$ partial difference set packings with respect to $(a_1,\dots,a_t)$ in an elementary abelian $2$-group, in order to construct Hadamard matrices having the balanced splittable property with respect to multiple disjoint submatrices simultaneously. Are there further examples of such partial difference set packings?
Constraints on the possible parameters $\delta, t, a_1, \dots, a_t$ are given in \cite[Theorem 1]{vDM}, expressed in the language of amorphic association schemes. 

\item
\cref{Res:Constructions}~$(vi)$ describes a direct construction (not involving a Kronecker product) of an infinite family of imprimitive balanced splittable Hadamard matrices. Are there further such constructions?

\end{enumerate}

\end{section}

\section*{Acknowledgements}

The authors are grateful to Hadi Kharaghani for kindly supplying the slides of his presentation ``Balancedly splittable orthogonal designs'' given at the \emph{8th European Congress of Mathematics} in 2021.
They thank John Polhill for helpful discussion and for pointing out the paper~\cite{vDM}.
They are grateful to the reviewers for their very careful reading of the manuscript.

\end{document}